\documentclass[a4paper,leqno,10pt]{article}

\usepackage{epsf,graphicx,epsfig}
\usepackage{amscd}
\usepackage{amsmath,latexsym,amssymb,amsthm}
\usepackage[nospace,noadjust]{cite}
\usepackage{textcomp}
\usepackage{setspace,cite}
\usepackage{lscape,fancyhdr,fancybox}
\usepackage{stmaryrd}
\usepackage[all,cmtip]{xy}
\usepackage{tikz}
\usepackage{cancel}
\usepackage{dsfont}
\usetikzlibrary{shapes,arrows,decorations.markings}
\usepackage{graphicx}

\usepackage{color}
\usepackage{url}
\usepackage{enumerate}
\usepackage[mathscr]{euscript}
  \theoremstyle{plain}

\swapnumbers
    \newtheorem{thm}{Theorem}[section]
    \newtheorem{prop}[thm]{Proposition}

    \newtheorem{subsec}[thm]{}
\theoremstyle{definition}
    \newtheorem{defn}[thm]{Definition}
        \newtheorem{remark}[thm]{Remark}
\theoremstyle{remark}

\title{}
\author{}
\date{}
\usepackage{amssymb}

\usepackage{hyperref}
\hypersetup{
	colorlinks,
	citecolor=blue,
	filecolor=black,
	linkcolor=blue,
	urlcolor=black
}

\begin{document}
\title{Compatible, split and family Loday-algebras}

% \footnote{2020 {\em Mathematics Subject Classification.} \textcolor{red}{LEFT}} \footnote{{\em Keywords.}  \textcolor{red}{LEFT} }}

\author{Apurba Das\\
Indian Institute of Technology, Kharagpur \\Kharagpur 721302, West Bengal, India. \\
Email: apurbadas348@gmail.com}

%\address{Indian Institute of Technology, Kharagpur 721302, West Bengal, India. Email: apurbadas348@gmail.com}
%~and Sourav Sen\footnote{Harish-Chandra Research Institute, HBNI, Jhunsi, Allahabad 211019, India. Email: sourav.sen3@gmail.com}
%}
%\email{apurbadas348@gmail.com} and Sourav Sen}
%\address{Department of Mathematics, Indian Institute of Technology, Kharagpur 721302, West Bengal,  India.}
%\email{apurbadas348@gmail.com}

%\author{Sourav Sen}
%\address{Harish-Chandra Research Institute, HBNI, Chhatnag Road, Jhunsi, Allahabad 211019, India.}
%\email{sourav.sen3@gmail.com}

%\curraddr{}
%\email{}

%\subjclass[2010]{}
%\keywords{}

%\begin{abstract}

%\end{abstract}

\maketitle

\noindent

\thispagestyle{empty}

%\begin{center}

%\end{center}

\noindent {\bf Abstract.} Given a nonsymmetric operad $\mathcal{O}$, we first construct two new nonsymmetric operads $\mathcal{O}^{\mathrm{comp}}$ and $\mathcal{O}^{\mathrm{Dend}}$. These operads are respectively useful to study compatible and split Loday-algebras. As an application of the operad $\mathcal{O}^{\mathrm{comp}}$, we show that the cohomology of a compatible associative algebra carries a Gerstenhaber structure. We give an application of the operad $\mathcal{O}^\mathrm{Dend}$ to dendriform algebras and find generalizations to other Loday-algebras. In the end, we construct another operad $\mathrm{Fam}(\mathcal{O}^\Omega)^\mathrm{Dend}$ to study dendriform-family algebras recently introduced in the literature. We also define and study homotopy dendriform-family algebras.\\



\tableofcontents

\vspace{0.2cm}

\section{Introduction}\label{sec-1}
Cohomology of algebraic structures is a very classical and powerful subject of research. This subject developed with the pioneer works of Hochschild, Serre, Chevalley, Eilenberg and Koszul \cite{hoch1,hoch2,chev,kosz}, among others. The cohomology of associative algebras is useful to study extensions of algebras. Later, Gerstenhaber introduced formal deformations of associative algebras and finds a significant connection with cohomology theory \cite{gers}. He also finds a cup-product and a degree $-1$ graded Lie bracket on cohomology which makes the graded space of cohomology a $G$-algebra (popularly known as Gerstenhaber algebra) structure. In \cite{gers-voro} Gerstenhaber and Voronov generalized this study to nonsymmetric operads equipped with multiplications and obtained a connection with Deligne's conjecture. Cohomologies for other algebraic structures are also developed over the years. Among others, Nijenhuis and Richardson \cite{nij-ric} developed cohomology of Lie algebras, Balavoine \cite{bala-leib} developed cohomology of Leibniz algebras etc. In  \cite{bala} Balavoive considers the cohomology theory of algebras over any symmetric, binary quadratic operad. Algebras whose defining identity/identities have no shuffling are called `Loday-algebras'. For example, associative algebras are Loday-algebras whereas Lie algebras are not. In \cite{das,yau} the authors showed that Loday-algebras can be described by multiplications in suitable nonsymmetric operads. Hence their cohomology automatically inherits a Gerstenhaber structure.

\medskip

Given a type of algebraic structure, three different kinds of structures are becoming popular in recent times. They are namely compatible algebraic structures, split algebraic structures and family algebraic structures \cite{aguiar,agu,loday-aguiar,cdm,A4,fard-bondia,kre,loday,Odesskii2,Odesskii3,yau,zhang,zhang-free,liu}. Since split and family algebraic structures are mostly studied for Loday-algebras, we will deal with them using multiplications on nonsymmetric operads. 
Our aim in this paper is to systematically study compatible Loday-algebras, split Loday-algebras and family Loday-algebras using nonsymmetric operads.

\subsection{Compatible Loday-algebras}

(Contents of Section \ref{sec-3}). Two algebraic structures of the same kind are said to be compatible if their sum also defines an algebraic structure of that kind. For example, two associative products $\cdot_1 : A \otimes A \rightarrow A$ and $\cdot_2 : A \otimes A \rightarrow A$ on a ${\bf k}$-module $A$ are said to be `compatible' if
\begin{align}\label{comp-ass-id}
(a \cdot_1 b) \cdot_2 c + (a \cdot_2 b) \cdot_1 c = a \cdot_1 (b \cdot_2 c) + a \cdot_2 (b \cdot_1 c), \text{ for } a, b, c \in A.
\end{align}
This can be equivalently described by the fact that the sum $\cdot_1 + \cdot_2$ defines a new associative product on $A$. In this case, the triple $(A, \cdot_1, \cdot_2)$ is called a compatible associative algebra. Compatible associative algebras are closely related with linear deformations and the classical Yang-Baxter equation \cite{Odesskii2,Odesskii3}. Recently, cohomology theory of compatible associative algebras has been defined in \cite{cdm} and deformations are studied. However, the structure of the cohomology ring of compatible associative algebras is not yet found. On the other hand, cohomologies of other compatible algebras are not even defined yet. By keeping these questions in mind, we first introduce a notion of `compatible multiplication' in a nonsymmetric operad. More precisely, a compatible multiplication is a pair $(\pi_1, \pi_2)$ consisting of two multiplications $\pi_1$ and $\pi_2$ satisfying a compatibility condition which is equivalent that their sum also defines a multiplication (see Definition \ref{defn-comp-mul}). Given a nonsymmetric operad $\mathcal{O}$, next we construct a new operad $\mathcal{O}^\mathrm{comp}$. A compatible multiplication on the operad $\mathcal{O}$ is equivalent to multiplication on the operad $\mathcal{O}^\mathrm{comp}$. Given a compatible multiplication $(\pi_1, \pi_2)$ on $\mathcal{O}$, we define the cohomology induced by $(\pi_1, \pi_2)$ as the cohomology induced by the corresponding multiplication on the operad $\mathcal{O}^\mathrm{comp}$. When $\mathcal{O}$ is the endomorphism operad $\mathrm{End}_A$ and compatible multiplication $(\pi_1, \pi_2)$ corresponds to the compatible associative algebra $(A, \cdot_1, \cdot_2)$, we obtain the cohomology of the compatible associative algebra introduced in \cite{cdm}. As a consequence of our study, we obtain a Gerstenhaber algebra structure on the cohomology of any compatible associative algebra. It is important to remark that compatible algebras can be studied over any symmetric, binary quadratic operad (see Remark \ref{remark-comp}).

\subsection{Split Loday-algebras}

(Contents of Section \ref{sec-4}). In \cite{loday} J.-L. Loday introduced a notion of dendriform algebra in his study of periodicity phenomenons in algebraic $K$-theory. A dendriform algebra is a  $(A, \prec, \succ)$ consisting of a ${\bf k}$-module $A$ with two bilinear maps $\prec, \succ : A \otimes A \rightarrow A$ satisfying
\begin{align}
(a \prec b) \prec c &= a \prec (b \prec c + b \succ c) \label{dend-1}\\
(a \succ b) \prec c &= a \succ (b \prec c), \label{dend-2}\\
(a \prec b + a \succ b) \succ c &= a \succ (b \succ c), \text{ for } a,b,c \in A. \label{dend-3}
\end{align}
\noindent If $(A, \prec, \succ)$ is a dendriform algebra, it turns out that the sum $\prec  + \succ$ is an associative product, called the total associative product. Thus, a dendriform algebra is a type of associative algebra whose product splits into two operations satisfying the dendriform identities (\ref{dend-1})-(\ref{dend-3}). Generalizing these identies in the context of nonsymmetric operads, we define a notion of dendriform-multiplication in a nonsymmetric operad $\mathcal{O}$. A dendriform-multiplication naturally splits a multiplication. Next, we construct a new operad $\mathcal{O}^\mathrm{Dend}$. We observe that a dendriform-multiplication on $\mathcal{O}$ is nothing but a multiplication on the operad $\mathcal{O}^\mathrm{Dend}$.
Using this, we define the cohomology induced by a dendriform-multiplication. This cohomology generalizes the standard cohomology of a dendriform algebra given in \cite{lod-val-book,A4,yau}.

Besides dendriform algebras, J.-L. Loday and his collaborators introduced some other algebras (e.g. diassociative, triassociative, tridendriform, quadri etc.) in their study of combinatorial algebras. All these algebras split associative algebras in the sense that the sum of the defining operations of these algebras forms an associative product. They are all Loday-algebras. Given any type `Lod' among the above mentioned Loday-algebras, we also construct an operad $\mathcal{O}^\mathrm{Lod}$ which allows us to define Lod-multiplication on $\mathcal{O}.$ In particular, when $\mathrm{Lod} = \mathrm{Dend}$, the type of dendriform algebras, we obtain dendriform-multiplication. We also write the explicit definition of tridendriform-multiplication by considering $\mathrm{Lod} = \mathrm{TriDend}$, the type of tridendriform algebras.

\subsection{Family Loday-algebras}

(Contents of Section \ref{sec-5}). Family algebraic structures relative to a semigroup $\Omega$ first appeared in the works of K. Ebrahimi-Fard, J. Gracia-Bondia and F. Patras in the algebraic formulations of renormalization in quantum field theory \cite{fard-bondia} (see also \cite{kre}). The notion of the Rota-Baxter family was also introduced in the same paper as the family analogue of the Rota-Baxter operator (we refer \cite{rota} for more details on Rota-Baxter operators). A Rota-Baxter family induces a dendriform-family structure (family analogue of dendriform structure) \cite{zhang,zhang-free}. Moreover, a dendriform-family structure gives rise to an associative algebra relative to $\Omega$ in the sense of Aguiar \cite{agu}. In Section \ref{sec-5}, we first construct a new nonsymmetric operad $\mathcal{O}^\Omega$ from a given operad $\mathcal{O}$. We observed that there is a nonsymmetric {suboperad} $\mathrm{Fam}(\mathcal{O}^\Omega)^\mathrm{Dend} \subset (\mathcal{O}^\Omega)^\mathrm{Dend}$, where $(\mathcal{O}^\Omega)^\mathrm{Dend}$ is obtained from $\mathcal{O}^\Omega$ by applying the construction $(~)^\mathrm{Dend}$ given in Section \ref{sec-4}. If $\mathcal{O} = \mathrm{End}_A$ is the endomorphism operas, a multiplication on the operad $\mathrm{Fam}(\mathrm{End}_A^\Omega)^\mathrm{Dend}$ is equivalent to having a dendriform-family structure on $A$. This characterization allows us to define the cohomology of a dendriform-family algebra.

Finally, we introduce $Dend_\infty$-family algebras as the homotopy analogue of dendriform-family algebras. We observed that a $Dend_\infty$-family algebra induces an ordinary $Dend_\infty$-algebra, thus generalising a result of \cite{zhang} in the homotopy context. W show that a $Dend_\infty$-algebra gives rise to an $A_\infty$-algebra relative to $\Omega$, which generalizes the result of Aguiar. In the end, we define the homotopy Rota-Baxter family that induces a ${Dend}_\infty$-family algebra.

\medskip

Throughout the paper, ${\bf k}$ is a commutative unital ring with characteristic $0$. All modules, (multi)linear maps or homomorphisms, tensor products are over ${\bf k}$. In Section \ref{sec-5}, we assume that $\Omega$ is a semigroup (not necessarily commutative and unital).

\section{Nonsymmetric operads}\label{sec-2}
In this section, we recall some basics of nonsymmetric operads. In particular, we mention the Gerstenhaber algebra structure on the cohomology induced by multiplications in a nonsymmetric operad. Our main references are \cite{gers-voro,lod-val-book,das}.

\begin{defn}
A {\bf nonsymmetric operad} $\mathcal{O}$ in the category of ${\bf k}$-modules consists of a collection $\{ \mathcal{O}(n) \}_{n \geq 1}$ of ${\bf k}$-modules together with ${\bf k}$-bilinear maps (called the partial compositions)
\begin{align*}
\circ_i : \mathcal{O}(m) \otimes \mathcal{O}(n) \rightarrow \mathcal{O}(m+n-1), \text{ for } m, n \geq 1 \text{ and } 1 \leq i \leq m
\end{align*}
and a distinguished element $\mathds{1} \in \mathcal{O}(1)$ (called the identity element) satisfying for $f \in \mathcal{O}(m)$, $g \in \mathcal{O}(n)$ and $h \in \mathcal{O}(p)$,
\begin{align}
& \begin{cases} (f \circ_i g) \circ_{i+j-1} h = f \circ_i (g \circ_j h), ~~ \text{ if } 1 \leq i \leq m, ~ 1 \leq j \leq n,\\
 (f \circ_i g) \circ_{j+n-1} h = (f \circ_j h) \circ_i g,~ \text{ if } 1 \leq i < j \leq m,
 \end{cases} \label{pc-iden}\\
& \qquad  f \circ_i \mathds{1} = \mathds{1} \circ_1 f = f, ~ \text{ if } 1 \leq i \leq m.
\end{align}
We often denote a nonsymmetric operad as above by the triple $\mathcal{O} = ( \{ \mathcal{O}(n) \}_{n \geq 1}, \circ_i, \mathds{1}).$
\end{defn}

\begin{defn}
Let $\mathcal{O} = ( \{ \mathcal{O}(n) \}_{n \geq 1}, \circ_i, \mathds{1})$ and $\mathcal{O}' = ( \{ \mathcal{O}'(n) \}_{n \geq 1}, \circ'_i, \mathds{1}')$ be two nonsymmetric operads. A {\bf morphism} $\varphi : \mathcal{O} \rightarrow \mathcal{O}'$ of nonsymmetric operads is collection $\varphi = \{ \varphi_n : \mathcal{O}(n) \rightarrow \mathcal{O}'(n) \}_{n \geq 1}$ of ${\bf k}$-linear maps satisfying
\begin{align*}
\varphi_{m+n-1} (f \circ_i g) = \varphi_m (f) \circ_i' \varphi_n (g) ~~ \text{ and } ~~ \varphi_1 (\mathds{1}) = \mathds{1}',~ \text{ for } f \in \mathcal{O}(m), g \in \mathcal{O}(n) \text{ and } 1 \leq i \leq m.
\end{align*}
\end{defn}

A toy example of a nonsymmetric operad is given by the endomorphism operad associated to a ${\bf k}$-module. Let $A$ be a ${\bf k}$-module. The endomorphism operad associated to $A$ is given by $\mathrm{End}_A = (   \{  \mathrm{End}_A (n)   \}_{n \geq 1}, \circ_i, \mathds{1}   )$, where
\begin{align*}
\mathrm{End}_A (n) := \mathrm{Hom}(A^{\otimes n}, A), \text{ for } n \geq 1
\end{align*}
and for any $f \in \mathrm{End}_A (m)$, $g \in \mathrm{End}_A (n)$ and $1 \leq i \leq m$, the partial composition $f \circ_i g$ is the substitution of $g$ in the $i$-th input of $f$. We will frequently use this endomorphism operad as the motivation of our various study.

A nonsymmetric operad naturally gives rise to a pre-Lie system in the sense of \cite{gers-voro}. Hence one can obtain a degree $-1$ graded Lie algebra from a nonsymmetric operad. More generally, let $\mathcal{O} = (  \{ \mathcal{O}(n) \}_{n \geq 1}, \circ_i, \mathds{1}   )$ be a nonsymmetric operad. Then the graded ${\bf k}$-module $\mathcal{O}(\bullet) := \oplus_{n \geq 1} \mathcal{O} (n)$ carries the following degree $-1$ graded Lie bracket (generalizing the Gerstenhaber bracket \cite{gers,gers2})
\begin{align}\label{g-bracket}
\llbracket f, g \rrbracket = \sum_{i=1}^m (-1)^{(n-1)(i-1)}~ f \circ_i g ~-~ (-1)^{(m-1)(n-1)} \sum_{i=1}^n (-1)^{(m-1)(i-1)} g \circ_i f , 
\end{align}
for $f \in \mathcal{O}(m)$ and $g \in \mathcal{O}(n)$. In other words, the shifted graded ${\bf k}$-module $\mathcal{O}(\bullet)[1] := \oplus_{n \geq 0} \mathcal{O}(n+1)$ equipped with the bracket (\ref{g-bracket}) is a graded Lie algebra.

\begin{defn}
(i) Let $\mathcal{O} = (  \{ \mathcal{O}(n) \}_{n \geq 1}, \circ_i, \mathds{1}   )$ be a nonsymmetric operad. An element $\pi \in \mathcal{O}(2)$ is said to be a {\bf multiplication} on $\mathcal{O}$ if
\begin{align*}
\pi \circ_1 \pi = \pi \circ_2 \pi.
\end{align*}

(ii) Let $\mathcal{O}$ and $\mathcal{O}'$ be two nonsymmetric operads. Let $\pi \in \mathcal{O}(2)$ be a multiplication on the operad $\mathcal{O}$ and $\pi' \in \mathcal{O}'(2)$ be a multiplication on the operad $\mathcal{O}'$. A {\bf morphism} from $\pi$ to $\pi'$ is given by a morphism $\varphi : \mathcal{O} \rightarrow \mathcal{O}'$ of nonsymmetric operads satisfying $\varphi_2 (\pi) = \pi'$.
\end{defn}

\begin{remark}
For any ${\bf k}$-module $A$, consider the endomorphism operad $\mathrm{End}_A$. Then an element $\pi \in \mathrm{End}_A (2)$ corresponds to a ${\bf k}$-bilinear product $\cdot : A \otimes A \rightarrow A$ defined by $a \cdot b = \pi (a, b)$, for $a, b \in A$. Then $\pi$ is a multiplication on the operad $\mathrm{End}_A$ if and only if the product $\cdot$ defines an associative algebra structure on $A$. Thus, multiplications on nonsymmetric operads generalize associative algebra structures on ${\bf k}$-modules.
\end{remark}

\begin{remark}
Let $\mathcal{O}$ be a nonsymmetric operad and $\pi \in \mathcal{O}(2)$ be a multiplication. Then we have from (\ref{g-bracket}) that
\begin{align*}
\llbracket \pi, \pi \rrbracket = 2 \big( \pi \circ_1 \pi - \pi \circ_2 \pi   \big) = 0.
\end{align*}
This shows that $\pi$ is a Maurer-Cartan element in the graded Lie algebra $\mathcal{O}(\bullet)[1] = ( \oplus_{n \geq 0} \mathcal{O}(n+1), \llbracket ~,~ \rrbracket )$. Conversely, if the characteristic of the ring ${\bf k}$ is other than $2$, then a Maurer-Cartan element of $\mathcal{O}(\bullet)[1]$ is nothing but a multiplication on the nonsymmetric operad $\mathcal{O}$.
\end{remark}

Any multiplication $\pi \in \mathcal{O}(2)$ on the operad $\mathcal{O}$ induces a (degree $0$) cup-product
\begin{align}\label{c-product}
f \smile_\pi g :=  (-1)^{mn+1} ~(\pi \circ_2 g)\circ_1 f, \text{ for } f \in \mathcal{O}(m), g \in \mathcal{O}(n).
\end{align}
Since $\pi$ is a multiplication, it follows that the cup-product $\smile_\pi$ is associative. When $\mathcal{O}$ is the endomorphism operad $\mathrm{End}_A$ associated to the ${\bf k}$-module $A$, and the multiplication $\pi \in \mathrm{End}_A(2)$ corresponds to the associative product $\cdot : A \otimes A \rightarrow A$, then the cup-product becomes
\begin{align*}
(f \smile_\pi g)(a_1, \ldots, a_{m+n}) = (-1)^{mn+1}~ f(a_1, \ldots, a_m) \cdot g(a_{m+1}, \ldots, a_{m+n}),
\end{align*}
for $f \in \mathrm{End}_A (m)$, $g \in \mathrm{End}_A(n)$ and $a_1, \ldots, a_{m+n} \in A$. This is precisely the Gerstenhaber's cup-product defined in \cite{gers,gers2} on the space of endomorphisms on $A$. Thus, the product (\ref{c-product}) generalizes the Gerstenhaber's cup-product.

A multiplication on a nonsymmetric operad also induces a cohomology theory which generalizes the Hochschild cohomology theory of associative algebras. Let $\mathcal{O}$ be a nonsymmetric operad and $\pi \in \mathcal{O}(2)$ be a multiplication. For each $n \geq 1$, we define a ${\bf k}$-module $C^n_\pi (\mathcal{O})$ simply by $C^n_\pi (\mathcal{O}) := \mathcal{O}(n)$ and a map $\delta_\pi : C^n_\pi (\mathcal{O}) \rightarrow C^{n+1}_\pi (\mathcal{O})$ by
\begin{align*}
\delta_\pi (f) := \llbracket \pi, f \rrbracket, \text{ for } f \in C^n_\pi (\mathcal{O}).
\end{align*}
Since $\llbracket \pi, \pi \rrbracket = 0$, it follows that $\delta_\pi \circ \delta_\pi = 0$. In other words, $(C^\bullet_\pi (\mathcal{O}), \delta_\pi)$ is a cochain complex. The corresponding cohomology groups are denoted by $H^\bullet_\pi (\mathcal{O})$ and called the {\bf cohomology induced by the multiplication $\pi$}.

\medskip

\begin{defn}
A {\bf Gerstenhaber algebra} is a triple $\big(  \mathcal{G} , \smile, \llbracket ~, ~ \rrbracket \big)$ consisting of a graded ${\bf k}$-module $ \mathcal{G} = \oplus_{n \in \mathbb{Z}} \mathcal{G}(n)  $ together with a graded-commutative associative product $\smile ~\! : \mathcal{G}  \otimes \mathcal{G}  \rightarrow \mathcal{G} $ and a degree $-1$ graded Lie bracket $\llbracket  ~, ~ \rrbracket : \mathcal{G} \otimes \mathcal{G} \rightarrow \mathcal{G} $ satisfying the following Leibniz rule
\begin{align}\label{l-rule}
\llbracket x, y \smile z \rrbracket = \llbracket x, y \rrbracket \smile y + (-1)^{(m-1)n}~ y \smile \llbracket x, z \rrbracket, \text{ for } x \in \mathcal{G}(m), y \in \mathcal{G}(n) \text{ and } z \in \mathcal{G}.
\end{align}
\end{defn}

In this case, we often say that the graded ${\bf k}$-module $\mathcal{G}$ carries a Gerstenhaber structure when the product and bracket are clear from the context. Gerstenhaber algebra appears in many places of mathematics and mathematical physics. In particular, the Hochschild cohomology ring inherits a Gerstenhaber algebra structure. The same result holds for the cohomology induced by multiplication in a nonsymmetric operad. Let $\mathcal{O}$ be a nonsymmetric operad and $\pi$ be a multiplication. The product (\ref{c-product}) induces an associative product (denoted by the same notation $\smile_\pi$) on the graded space $H^\bullet_\pi (\mathcal{O})$ of cohomology. The induced product turns out to be graded-commutative on the level of cohomology. The degree $-1$ graded Lie bracket (\ref{g-bracket}) also passes onto the cohomology. Moreover, the induced product and degree $-1$ graded Lie bracket on the cohomology $H^\bullet_\pi (\mathcal{O})$ satisfy further the Leibniz rule (\ref{l-rule}). Hence one obtains the following result.

\begin{thm}
Let $\mathcal{O}$ be a nonsymmetric operad and $\pi \in \mathcal{O}(2)$ be a multiplication. Then the graded space $H^\bullet_\pi (\mathcal{O})$ of cohomology inherits a Gerstenhaber algebra structure.
\end{thm}

\medskip

Let $\mathcal{O}, \mathcal{O}'$ be two nonsymmetric operads and $\pi, \pi'$ be two multiplications on the operads $\mathcal{O}$ and $\mathcal{O}'$, respectively. If $\varphi : \mathcal{O} \rightarrow \mathcal{O}'$ is a morphism between multiplications from $\pi$ to $\pi'$, then $\phi$ induces a morphism of cochain complexes from $(C^\bullet_\pi (\mathcal{O}), \delta_\pi)$ to $(C^\bullet_{\pi'} (\mathcal{O}), \delta_{\pi'})$. The induced map preserves the cup-product and the degree $-1$ graded Lie bracket. Therefore, one obtains a morphism $\varphi_\ast : H^\bullet_\pi (\mathcal{O}) \rightarrow H^\bullet_{\pi'} (\mathcal{O}')$ between cohomology groups which is also a morphism between Gerstenhaber algebras.

\section{Compatible multiplications and the operad $\mathcal{O}^\mathrm{comp}$}\label{sec-3}

In this section, we define the notion of `compatible multiplication' on a nonsymmetric operad $\mathcal{O}$. We construct the operad $\mathcal{O}^\mathrm{comp}$ to study compatible multiplications on the operad $\mathcal{O}$. We also define the cohomology induced by a compatible multiplication on $\mathcal{O}$. As an application of our study, we conclude that the cohomology of a compatible associative algebra carries a Gerstenhaber structure.

%\subsection{Compatible multiplications}

\begin{defn}\label{defn-comp-mul}
(i) Let $\pi_1$ and $\pi_2 $ be two multiplications on the nonsymmetric operad $\mathcal{O}$. They are said to be {\bf compatible} if
\begin{align}\label{comp}
\pi_1 \circ_1 \pi_1 + \pi_2 \circ_1 \pi_1 = \pi_1 \circ_2 \pi_2 + \pi_2 \circ_2 \pi_1.
\end{align} 

(ii) A {\bf compatible multiplication} on the nonsymmetric operad $\mathcal{O}$ is a pair $(\pi_1, \pi_2)$ consisting of two multiplications $\pi_1$ and  $\pi_2$ that are compatible.
\end{defn}

Note that the compatibility condition (\ref{comp}) is equivalent to the fact that the linear combination $\lambda \pi_1 + \mu \pi_2$ is a multiplication on $\mathcal{O}$, for any $\lambda, \mu \in {\bf k}.$ On the other hand, using the expression (\ref{g-bracket}), we have
\begin{align*}
\llbracket \pi_1, \pi_2 \rrbracket = \pi_1 \circ_1 \pi_2 - \pi_1 \circ_2 \pi_2 + \pi_2 \circ_1 \pi_1 - \pi_2 \circ_2 \pi_1.
\end{align*}
Thus, it follows that $\pi_1$ and $\pi_2$ are compatible if and only if $\llbracket \pi_1, \pi_2 \rrbracket = 0$.

\begin{remark}
Let $A$ be a ${\bf k}$-module and $\mathrm{End}_A$ be the endomorphism operad associated to $A$. Let $\pi_1, \pi_2 \in \mathrm{End}_A (2)$ be two elements which correspond to ${\bf k}$-bilinear products $\cdot_1,~ \cdot_2 : A \otimes A \rightarrow A$, respectively. Then we know that $\pi_1$ (resp. $\pi_2$) is a multiplication on the operad $\mathrm{End}_A$ if and only if ~$\cdot_1$ (resp. $\cdot_2$ ) defines an associative product on $A$. Finally, the compatibility condition (\ref{comp}) is equivalent to the compatibility condition (\ref{comp-ass-id}) among associative products. Therefore, compatible multiplications on a nonsymmetric operad is a generalization of compatible associative algebras.
\end{remark}

% Then
%\begin{align*}
%\pi_1 \text{ is a multiplication on } \mathrm{End}_A &\Longleftrightarrow (A, \pi_1) \text{ is an associative algebra},\\
%\pi_2 \text{ is a multiplication on } \mathrm{End}_A &\Longleftrightarrow (A, \pi_2) \text{ is an associative algebra},\\
%\text{the compatibility condition } (\ref{comp}) &\Longleftrightarrow \text{ associative products } \pi_1, \pi_2 \text{ are compatible in the sence } (\textcolor{red}{LEFT}).
%\end{align*}

\medskip

Let $\mathcal{O} = ( \{ \mathcal{O}(n) \}_{n \geq 1}, \circ_i, \mathds{1})$ be a nonsymmetric operad. For each $n \geq 1$, we consider the ${\bf k}$-module 
\begin{align*}
\mathcal{O}^\mathrm{comp}(n) := \underbrace{\mathcal{O}(n) \oplus \cdots \oplus \mathcal{O}(n)}_{n \text{ copies}}.
\end{align*}
We define ${\bf k}$-bilinear maps $\circ_i^\mathrm{comp} : \mathcal{O}^\mathrm{comp}(m) \otimes \mathcal{O}^\mathrm{comp}(n) \rightarrow \mathcal{O}^\mathrm{comp}(m+n-1)$, for $1 \leq i \leq m$, by
\begin{align*}
f \circ_i^\mathrm{comp} g := \big( f_1 \circ_i g_1, \ldots, \underbrace{\sum_{r+s = k+1} f_r \circ_i g_s }_{k\text{-th place}} , \ldots, f_m \circ_i g_n \big), 
\end{align*}
for $f= (f_1, \ldots, f_m) \in \mathcal{O}^\mathrm{comp}(m)$ and $g = (g_1, \ldots, g_m) \in \mathcal{O}^\mathrm{comp}(n)$. Then we have the following.

\begin{thm}\label{comp-new-operad}
The triple $\mathcal{O}^\mathrm{comp} = (   \{ \mathcal{O}^\mathrm{comp} (n) \}_{n \geq 1}, \circ_i^\mathrm{comp}, \mathds{1})$ is a nonsymmetric operad. Moreover, an element  $\pi = (\pi_1, \pi_2) \in \mathcal{O}^\mathrm{comp}(2)$ is a multiplication on the operad $\mathcal{O}^\mathrm{comp}$ if and only if the pair $(\pi_1, \pi_2)$ of elements of $\mathcal{O}(2)$ forms a compatible multiplication on $\mathcal{O}$.
\end{thm}

\begin{proof}
Let $ f = (f_1, \ldots, f_m) \in \mathcal{O}^\mathrm{comp}(m)$, $ {g} = (g_1, \ldots, g_m) \in \mathcal{O}^\mathrm{comp}(n)$ and ${h} = (h_1, \ldots, h_p) \in \mathcal{O}^\mathrm{comp}(p)$. Then for $1 \leq i \leq m$ and $1 \leq j \leq n$, we have
\begin{align*}
&(f \circ_i^\mathrm{comp} g) \circ^\mathrm{comp}_{i+j-1} h \\
&= \big(  f_1 \circ_i g_1, \ldots, \underbrace{\sum_{r+s=k+1} f_r \circ_i g_s}_{k\text{-th place}}, \ldots, f_m \circ_i g_n \big) \circ_{i+j-1} (h_1, \ldots, h_p) \\
&= \big( (f_1 \circ_i g_1) \circ_{i+j-1} h_1, \ldots, \underbrace{  \sum_{r+s+t = k+2} (f_r \circ_i g_s) \circ_{i+j-1} h_t }_{k\text{-th place}}, \ldots, (f_m \circ_i g_n) \circ_{i+j-1} h_p   \big) \\
&= \big( f_1 \circ_i (g_1 \circ_j h_1), \ldots, \underbrace{\sum_{r+s+t = k+2} f_r \circ_i (g_s \circ_j h_t) }_{k\text{-th place}}, \ldots, f_m \circ_i (g_n \circ_j h_p)  \big) \\
&= f \circ_i^\mathrm{comp} (g \circ_j^\mathrm{comp} h).
\end{align*}
Similarly, for $1 \leq i < j \leq m$, we have
\begin{align*}
&(f \circ_i^\mathrm{comp} g) \circ_{j+n-1}^\mathrm{comp} h \\
&= (f_1 \circ_i g_1, \ldots, \underbrace{\sum_{r+s = k+1} f_r \circ_i g_s}_{k\text{-th place}}, \ldots, f_m \circ_i g_n) \circ_{j+n-1}^\mathrm{comp} (h_1, \ldots, h_p) \\
&= \big( (f_1 \circ_i g_1) \circ_{j+n-1} h_1, \ldots, \underbrace{ \sum_{r+s+t = k+2} (f_r \circ_i g_s) \circ_{j+n-1} h_t}_{k\text{-th place}}, \ldots, (f_m \circ_i g_n) \circ_{j+n-1} h_p \big) \\
&= \big(  (f_1 \circ_j h_1) \circ_i g_1 , \ldots, \underbrace{\sum_{r+s+t = k+2} (f_r \circ_j h_t) \circ_i g_s}_{k\text{-th place}}, \ldots, (f_m \circ_j h_p) \circ_i g_n \big) \\
&= (f \circ_j^\mathrm{comp} h) \circ_i^\mathrm{comp} g.
\end{align*}
Finally, for any $f = (f_1, \ldots, f_m) \in \mathcal{O}^\mathrm{comp}(m)$ and $1 \leq i \leq m$, we have
\begin{align*}
f \circ_i^\mathrm{comp} \mathds{1} = (f_1, \ldots, f_m) \circ_i^\mathrm{comp} \mathds{1} = ( f_1 \circ_i \mathds{1}, \ldots, f_m \circ_i \mathds{1}) = (f_1, \ldots, f_m) = f
\end{align*}
and
\begin{align*}
\mathds{1} \circ_1^\mathrm{comp} f = \mathds{1} \circ_1^\mathrm{comp} (f_1, \ldots, f_m) = ( \mathds{1} \circ_1 f_1, \ldots, \mathds{1} \circ_1 f_m) = (f_1, \ldots, f_m) = f.
\end{align*}
This proves that $(   \{ \mathcal{O}^\mathrm{comp} (n) \}_{n \geq 1}, \circ_i^\mathrm{comp}, \mathds{1})$ is a nonsymmetric operad.

\medskip

For the second part, we observe that
\begin{align*}
&\pi \circ^\mathrm{comp}_1 \pi ~-~ \pi \circ_2^\mathrm{comp} \pi \\
&= (\pi_1 \circ_1 \pi_1,~ \pi_1 \circ_1 \pi_2 +\pi_2 \circ_1 \pi_1,~ \pi_2 \circ_1 \pi_2) ~-~ (\pi_1 \circ_2 \pi_1,~ \pi_1 \circ_2 \pi_2 +\pi_2 \circ_2 \pi_1,~ \pi_2 \circ_2 \pi_2) \\
&= \big(  \pi_1 \circ_1 \pi_1 - \pi_1 \circ_2 \pi_1, ~ \pi_1 \circ_1 \pi_2 +\pi_2 \circ_1 \pi_1 - \pi_1 \circ_2 \pi_2 - \pi_2 \circ_2 \pi_1 ,~\pi_2 \circ_1 \pi_2 - \pi_2 \circ_2 \pi_2  \big).
\end{align*}
This shows that $\pi$ is a multiplication on the operad $\mathcal{O}^\mathrm{comp}$ if and only if $(\pi_1, \pi_2)$ is a compatible multiplication on the operad $\mathcal{O}$.
\end{proof}

\medskip

Let $(\pi_1, \pi_2)$ be a compatible multiplication on the operad $\mathcal{O}$. Then we have seen that $\pi = (\pi_1, \pi_2)$ is an (ordinary) multiplication on the operad $\mathcal{O}^\mathrm{comp}$. Hence we may consider the cohomology induced by the multiplication $\pi = (\pi_1, \pi_2)$ on the operad $\mathcal{O}^\mathrm{comp}$ (see Section \ref{sec-2}). More precisely, we consider the cochain complex $\big( C^\bullet_{(\pi_1, \pi_2)} (\mathcal{O}^\mathrm{comp}) , \delta_{(\pi_1, \pi_2)} \big)$, where
\begin{align*}
C^n_{(\pi_1, \pi_2)} (\mathcal{O}^\mathrm{comp}) := \mathcal{O}^\mathrm{comp} (n) = \underbrace{\mathcal{O}(n) \oplus \cdots \oplus \mathcal{O}(n)}_{n \text{ times}}, \text{ for } n \geq 1.
\end{align*}
Before we write the differential $\delta_{(\pi_1, \pi_2)}$, we observe that the degree $-1$ graded Lie bracket on $C^\bullet_{(\pi_1, \pi_2)} (\mathcal{O}^\mathrm{comp})$ is given by
\begin{align}\label{double-lie}
\llbracket (f_1, \ldots, f_m) , (g_1, \ldots, g_n ) \rrbracket_{\mathcal{O}^\mathrm{comp}} =~&  \sum_{i=1}^m (-1)^{(i-1)(n-1)} ~(f_1, \ldots, f_m) \circ_i^\mathrm{comp} (g_1, \ldots, g_n ) \\
& \qquad - (-1)^{(m-1) (n-1) } \sum_{i=1}^n (-1)^{(i-1)(m-1)} ~ (g_1, \ldots, g_n) \circ_i^\mathrm{comp} (f_1, \ldots, f_m ) \nonumber \\
=~& \big( \llbracket f_1, g_1 \rrbracket , \ldots, \underbrace{\sum_{r+s = k+1} \llbracket f_r , g_s \rrbracket }_{k\text{-th place}}, \ldots, \llbracket f_m , g_n \rrbracket \big), \nonumber
\end{align}
for $(f_1, \ldots, f_m) \in C^m_{(\pi_1, \pi_2)} (\mathcal{O}^\mathrm{comp})$ and $(g_1, \ldots, g_n) \in C^n_{(\pi_1, \pi_2)} (\mathcal{O}^\mathrm{comp})$.
With the above expression of the bracket, the differential $\delta_{(\pi_1, \pi_2)}$ is given by 
\begin{align}\label{diff-diff}
\delta_{(\pi_1, \pi_2)} (f_1, \ldots, f_n) := \llbracket (\pi_1, \pi_2), (f_1, \ldots, f_n) \rrbracket_{\mathcal{O}^\mathrm{comp}} = \big(  \delta_{\pi_1} (f_1), \ldots, \underbrace{\delta_{\pi_1} (f_k) + \delta_{\pi_2} (f_{k-1})}_{k\text{-th place}}, \ldots, \delta_{\pi_2} (f_n)  \big),
\end{align}
for $(f_1, \ldots, f_n) \in C^n_{(\pi_1, \pi_2)} (\mathcal{O}^\mathrm{comp})$. The cohomology of the cochain complex $\big( C^\bullet_{(\pi_1, \pi_2)} (\mathcal{O}^\mathrm{comp}) , \delta_{(\pi_1, \pi_2)} \big)$ is called the {\bf cohomology of the compatible multiplication} $(\pi_1, \pi_2)$ on the operad $\mathcal{O}$. We denote the corresponding cohomology groups by $H^\bullet_{(\pi_1, \pi_2)} (\mathcal{O}^\mathrm{comp}).$
%This cochain complex is same with the complex $(C^\bullet_{\pi, \pi'}(\mathcal{O}), \delta_{\pi, \pi'})$ defined for compatible multiplications $\pi, \pi'$. 

\medskip

Observe that the graded space $C^\bullet_{(\pi_1, \pi_2)} (\mathcal{O}^\mathrm{comp})$ also carries a cup-product induced by the multiplication $\pi = (\pi_1, \pi_2) \in \mathcal{O}^\mathrm{comp}(2)$. The cup-product is explicitly given by
\begin{align}\label{double-cup}
(f_1, \ldots, f_m) \smile_{(\pi_1, \pi_2)} (g_1, \ldots, g_n) = \big(  f_1 \smile_{\pi_1} g_1, \ldots, \underbrace{\sum_{r+s=k+1} f_r \smile_{\pi_1} g_s + \sum_{r+s=k} f_r \smile_{\pi_2} g_s}_{k\text{-th place}}, \ldots, f_m \smile_{\pi_2} g_n   \big),
\end{align}
for $(f_1, \ldots, f_m) \in C^m_{(\pi_1, \pi_2)} (\mathcal{O}^\mathrm{comp})$ and $(g_1, \ldots, g_n) \in C^n_{(\pi_1, \pi_2)} (\mathcal{O}^\mathrm{comp})$. As a summary, we get the following.

\begin{thm}
Let $(\pi, \pi')$ be a compatible multiplication on the nonsymmetric operad $\mathcal{O}$. Then the graded space of cohomology $H^\bullet_{(\pi_1, \pi_2)} (\mathcal{O}^\mathrm{comp})$ inherits a Gerstenhaber algebra structure.
\end{thm}

\begin{remark}
Let $\mathcal{O}$ be the endomorphism operad $\mathrm{End}_A$ associated to the ${\bf k}$-module $A$. Let $(\pi_1, \pi_2)$ be a compatible multiplication on the operad $\mathrm{End}_A$ which corresponds to the compatible associative algebra $(A, \cdot_1, \cdot_2)$. It follows from the above discussions that the cohomology induced by the compatible multiplication $(\pi_1, \pi_2)$ is given by the cohomology of the cochain complex $(C^\bullet_{(\pi_1, \pi_2)} (\mathrm{End}_A^\mathrm{comp}), \delta_{(\pi_1, \pi_2)})$, where
\begin{align*}
C^n_\mathrm{(\pi_1, \pi_2)} (\mathrm{End}_A^\mathrm{comp})  = \underbrace{\mathrm{Hom}(A^{\otimes n}, A) \oplus \cdots \oplus \mathrm{Hom}(A^{\otimes n}, A)}_{n \text{~ copies}}, \text{ for } n \geq 1
\end{align*}
and $\delta_{(\pi_1, \pi_2)} : C^n_\mathrm{(\pi_1, \pi_2)} (\mathrm{End}_A^\mathrm{comp}) \rightarrow C^{n+1}_\mathrm{(\pi_1, \pi_2)} (\mathrm{End}_A^\mathrm{comp})$ is given by the formula (\ref{diff-diff}). Here $\delta_{\pi_1}$ and $\delta_{\pi_2}$ denote the Hochschild coboundary operators for the associative algebras $(A, \cdot_1)$ and $(A, \cdot_2)$, respectively. This is the cochain complex for the cohomology of the compatible associative algebra $(A, \cdot_1, \cdot_2)$ considered in \cite{cdm}. As a conclusion, the cohomology of the compatible associative algebra $(A, \cdot_1, \cdot_2)$ inherits a Gerstenhaber structure. The explicit cup-product and degree $-1$ graded Lie bracket can be easily obtained from (\ref{double-cup}) and (\ref{double-lie}), respectively.
\end{remark}

%(ii) Let $(A, \cdot_1, \cdot_2)$ be a compatible associative algebra. Then the graded space of cohomology $H^\bullet_{\mathrm{compAss}} (A,A)$ inherits a Gerstenhaber structure.

\medskip

Let $\mathcal{O}$ be a nonsymmetric operad. Consider the operad $\mathcal{O}^\mathrm{comp} = ( \{ \mathcal{O}^\mathrm{comp} (n) \}_{n \geq 1}, \circ_i^\mathrm{comp}, \mathds{1})$ given in Theorem \ref{comp-new-operad}. We define a collection $\varphi = \{ \varphi_n : \mathcal{O}^\mathrm{comp} (n) \rightarrow \mathcal{O}(n) \}_{n \geq 1}$ of ${\bf k}$-linear maps by
\begin{align*}
\varphi_n (f_1, \ldots, f_n) = f_1 + \cdots + f_n, ~\text{ for } (f_1, \ldots, f_n) \in \mathcal{O}^\mathrm{comp} (n).
\end{align*}

\begin{prop}\label{comp-mor}
With the above notation, $\varphi :  \mathcal{O}^\mathrm{comp} \rightarrow \mathcal{O}$ is a morphism of nonsymmetric operads.
\end{prop}

\begin{proof}
For $f = (f_1, \ldots, f_m) \in \mathcal{O}^\mathrm{comp} (m)$, $g = (g_1, \ldots, g_n) \in \mathcal{O}^\mathrm{comp} (n)$ and $1 \leq i \leq m$, we have
\begin{align*}
\varphi (f \circ_i^\mathrm{comp} g) =~& \varphi \big( f_1 \circ_i g_1, \ldots, \underbrace{ \sum_{r+s = k+1} f_r \circ_i g_s}_{k\text{-th place}} , \ldots, f_m \circ_i g_n  \big) \\
=~& (f_1 \circ_i g_1) + \cdots + (\sum_{r+s = k+1} f_r \circ_i g_s) + \cdots + (f_m \circ_i g_n ) \\
=~& \sum_{k =1}^{m+n-1} \sum_{r+s = k+1} f_r \circ_i g_s \\
=~& (f_1 + \cdots + f_m) \circ_i (g_1 + \cdots + g_n) = \varphi_m (f) \circ_i \varphi_n (g).
\end{align*}
Finally, we have $\varphi_1 (\mathds{1}) = \mathds{1}$. Therefore, $\phi$ is a morphism of nonsymmetric operads.
\end{proof}
%\subsection{Application to compatible associative algebras}

\medskip

Let $(\pi_1, \pi_2)$ be a compatible multiplication on the operad $\mathcal{O}$. Then it can be seen as a multiplication on the operad $\mathcal{O}^\mathrm{comp}$. Note that $\varphi : \mathcal{O}^\mathrm{comp} \rightarrow \mathcal{O}$ is a morphism of nonsymmetric operads (see Proposition \ref{comp-mor}) such that $\varphi ((\pi_1, \pi_2)) = \pi_1 + \pi_2$. Since $(\pi_1, \pi_2)$ is a multiplication on the operad $\mathcal{O}^\mathrm{comp}$ and $\pi_1 + \pi_2$ is a multiplication on the operad $\mathcal{O}$, it follows that $\varphi : \mathcal{O}^\mathrm{comp} \rightarrow \mathcal{O}$ is a morphism between multiplications from $(\pi_1, \pi_2)$ to $\pi_1 + \pi_2$. Therefore, as a consequence, $\varphi$ induces a morphism 
\begin{align*}
\varphi_\ast : H^\bullet_{(\pi_1, \pi_2)} (\mathcal{O}^\mathrm{comp}) \rightarrow H^\bullet_{\pi_1 + \pi_2} (\mathcal{O})
\end{align*}
from the cohomology induced by the compatible multiplication $(\pi_1, \pi_2)$ to the cohomology induced by the multiplication $\pi_1 + \pi_2$. This is infact a morphism between Gerstenhaber algebras.

\begin{remark}\label{remark-comp}
As mentioned in the introduction, nonsymmetric operads are only useful to study algebras whose defining identity/identities have no shufflings. To deal with algebras with shufflings, one need to consider symmetric operads. A symmetric operad is a nonsymmetric operad $\mathcal{O}$ equipped with certain actions of symmetric groups that are compatible with partial compositions. The endomorphism operad $\mathrm{End}_A$ (associated to a ${\bf k}$-module $A$) is a symmetric operad where the actions are given by shuffling of input entries. An algebra over a symmetric operad $\mathcal{O}$ is a {\bf k}-module $A$ equipped with a morphism $\mathcal{O} \rightarrow \mathrm{End}_A$ of symmetric operads. In particular, if $\mathcal{O}$ is nonsymmetric (consider as a symmetric operad with trivial actions of symmetric groups), Koszul and quadratic operad, an algebra over $\mathcal{O}$ can be equivalently described by multiplication on a suitable nonsymmetric operad constructed from $\mathcal{O}$ (see \cite{lod-val-book,das}). Thus, algebras over symmetric operads are more general notions than multiplications on nonsymmetric operads.

In \cite{bala} Balavoine observed that an algebra over a symmetric, binary and quadratic operad $\mathcal{O}$ can be regarded as a Maurer-Cartan element in a suitably graded Lie algebra. By considering compatible Maurer-Cartan elements in a graded Lie algebra, one can easily deal with compatible algebras over any symmetric, binary and quadratic operad.
\end{remark}

\section{Dendriform-multiplications and the operad $\mathcal{O}^\mathrm{Dend}$}\label{sec-4}
In this section, we first introduce a notion of dendriform-multiplication on a nonsymmetric operad $\mathcal{O}$. Such a multiplication splits a multiplication on $\mathcal{O}$. When we consider the endomorphism operad $\mathrm{End}_A$, a dendriform-multiplication corresponds to a dendriform algebra structure on $A$. Next, we construct a nonsymmetric operad $\mathcal{O}^\mathrm{Dend}$ whose ordinary multiplications are precisely dendriform-multiplications on $\mathcal{O}$. These ideas are further generalized to some other Loday-multiplications.

\subsection{Dendriform-multiplications}\label{subsec-dend-mul} 

\begin{defn}\label{defn-dend-mul}
Let $\mathcal{O} = ( \{  \mathcal{O}(n) \}_{n \geq 1}, \circ_i, \mathds{1})$ be a nonsymmetric operad. A {\bf dendriform-multiplication} on $\mathcal{O}$ consists of a pair $(\pi_\prec ,  \pi_\succ)$ of elements of $\mathcal{O}(2)$ satisfying
\begin{align}
\pi_\prec \circ_1 \pi_\prec =~& \pi_\prec \circ_2 (\pi_\prec + \pi_\succ), \label{dend-mul-1}\\
\pi_\prec \circ_1 \pi_\succ =~& \pi_\succ \circ_2 \pi_\prec,\\
\pi_\succ \circ_1 (\pi_\prec + \pi_\succ) =~& \pi_\succ \circ_2 \pi_\succ. \label{dend-mul-3}
\end{align}
\end{defn}

It follows from the above definition that neither $\pi_\prec$ nor $\pi_\succ$ is a multiplication on $\mathcal{O}$. However, if we add the identities (\ref{dend-mul-1})-(\ref{dend-mul-3}), we simply get that the sum $\pi_\mathrm{Tot} := \pi_\prec + \pi_\succ$ is a multiplication on $\mathcal{O}$. This is called the 'total' multiplication induced by the dendriform-multiplication.

\begin{remark}
Consider the endomorphism operad $\mathrm{End}_A$ associated to a ${\bf k}$-module $A$. Let $\pi_\prec, \pi_\succ \in \mathrm{End}_A (2)$ be two elements which correspond to ${\bf k}$-bilinear maps $\prec, \succ : A \otimes A \rightarrow A$. Then $(\pi_\prec, \pi_\succ)$ is a dendriform-multiplication on the operad $\mathrm{End}_A$ if and only if $(A, \prec, \succ)$ is a dendriform algebra. Thus, dendriform-multiplications are generalization of dendriform algebra structures.
\end{remark}

\begin{defn}
Let $\pi \in \mathcal{O}(2)$ be a multiplication on a nonsymmetric operad $\mathcal{O}$. An element $R \in \mathcal{O}(1)$ is said to be a {\bf Rota-Baxter element} with respect to the multiplication $\pi$ if
\begin{align}\label{rb-ele}
( \pi \circ_2 R) \circ_1 R = R \circ_1 (\pi \circ_1 R ~+~ \pi \circ_2 R ).
\end{align}
\end{defn}

Rota-Baxter elements are generalization of the Rota-Baxter operator (see \cite{aguiar} for instance). Like an ordinary Rota-Baxter operator induces a dendriform algebra structure \cite{aguiar}, a Rota-Baxter element induces a dendriform-multiplication.

\begin{prop}
Let $R \in \mathcal{O}(1)$ be a Rota-Baxter element with respect to a multiplication $\pi$. Then the pair $(\pi_\prec, \pi_\succ)$ is a dendriform-multiplication on $\mathcal{O}$, where  $\pi_\prec = \pi \circ_2 R$ and $\pi_\succ = \pi \circ_1 R$.
\end{prop}

\begin{proof}
The proof relies on the use of (\ref{pc-iden}) and the identity (\ref{rb-ele}). We have
\begin{align*}
\pi_\prec \circ_1 \pi_\prec =~& (\pi \circ_2 R) \circ_1 (\pi \circ_2 R) \\
=~& \pi \circ_2 ((\pi \circ_2 R) \circ_1 R) \\
=~& \pi \circ_2  \big( R \circ_1 (\pi \circ_1 R ~+~ \pi \circ_2 R)  \big)\\
=~& (\pi \circ_2 R) \circ_2 (\pi \circ_1 R ~+~ \pi \circ_2 R) 
= \pi_\prec \circ_2 (\pi_\prec + \pi_\succ),
\end{align*} 
\begin{align*}
\pi_\prec \circ_1 \pi_\succ =~& (\pi \circ_2 R) \circ_1 (\pi \circ_1 R) \\
=~& (\pi \circ_1 R) \circ_2 (\pi \circ_2 R) = \pi_\succ \circ_2 \pi_\prec
\end{align*}
and
\begin{align*}
\pi_\succ \circ_1 (\pi_\prec + \pi_\succ) =~& (\pi \circ_1 R) \circ_1 (\pi \circ_1 R ~+~ \pi \circ_2 R) \\
=~& \pi \circ_1 (R \circ_1 (\pi \circ_1 R ~+~ \pi \circ_2 R)) \\
=~& \pi \circ_1 ((\pi \circ_2 R) \circ_1 R) \\
=~& (\pi \circ_1 R) \circ_2 (\pi \circ_1 R) = \pi_\succ \circ_2 \pi_\succ.
\end{align*}
Hence the result follows.
\end{proof}
%\subsection{The operad $\mathcal{O}^\mathrm{Dend}$}

In the following, we construct the operad $\mathcal{O}^\mathrm{Dend}$. For that, we need some combinatorial maps on the subsets of natural numbers. Let $C_n$ $(n \geq 1)$ be the set of first $n$ natural numbers. Since we will not allow additions (or any operations) on the elements of $C_n$, we usually write the elements of $C_n$ inside a paranthesis, i.e., we write $C_n = \{ [1], [2], \ldots, [n] \}$. For any $m,n \geq 1$ and $1 \leq i \leq m$, let us put the elements of $C_{m+n-1}$ into $m$ many boxes as follows:

%\textcolor{red}{DIAGRAM}

\begin{align*}
\framebox{$[1]$} \quad \framebox{$[2]$} \quad \cdots \quad \framebox{$[i-1]$} \quad \framebox{$[i]$\quad$\cdots$ $[i+n-2]$ ~ $[i+n-1]$} \quad \framebox{$[i+n]$} \quad \cdots \quad \framebox{$[m+n-1]$}~.
\end{align*}

\medskip

\noindent In other words, we put exactly one element in the first $i-1$ boxes, $n$ elements in the $i$-th box and again exactly one element in the last $m-i$ many boxes. Moreover, the elements are situated in the increasing order. We define maps $R_0  (m;1, \ldots, n, \ldots, 1) : C_{m+n-1} \rightarrow C_m$ and  $R_i (m;1, \ldots, n, \ldots, 1) : C_{m+n-1} \rightarrow {\bf k}[C_n]$ as follows: for $[r] \in C_{m+n-1}$,
\begin{align}
R_0  (m;1, \ldots, n, \ldots, 1) [r] :=~& [k] ~~~ \text{ if } [r] \text{ lies in the $k$-th box}, \label{r0}\\
R_i  (m;1, \ldots, n, \ldots, 1) [r] :=~& \begin{cases} [r-i+1] ~~~ \text{ if } [r] \text{ lies in the $i$-th box}, \\
[1] + \cdots + [n] ~~~~ \text{ otherwise}.
\end{cases} \label{ri}
\end{align}

%We are now in a position to construct the operad $\mathcal{O}^\mathrm{Dend}$. 
Let $\mathcal{O} = (\{ \mathcal{O}(n) \}_{n \geq 1}, \circ_i, \mathds{1})$ be a nonsymmetric operad. For each $n \geq 1$, we define a ${\bf k}$-module 
\begin{align*}
\mathcal{O}^\mathrm{Dend} (n) := {\bf k}[C_n] \otimes \mathcal{O}(n).
\end{align*}
 Any element $f \in \mathcal{O}^\mathrm{Dend}(n)$ can be regarded as an $n$-tuple $f = (f^{[1]}, \ldots, f^{[n]})$, where $f^{[1]}, \ldots, f^{[n]} \in \mathcal{O}(n)$. This can be seen by the following identification
\begin{align*}
f = (f^{[1]}, \ldots, f^{[n]}) \leftrightsquigarrow ~[1] \otimes f^{[1]} ~+ \cdots +~ [n] \otimes f^{[n]} \in \mathcal{O}^\mathrm{Dend}(n).
\end{align*}
(Note that, the {\bf k}-modules $\mathcal{O}^\mathrm{Dend} (n)$ and $\mathcal{O}^\mathrm{comp} (n)$ are isomorphic via the above identification.) For any $m, n \geq 1$ and $1 \leq i \leq m$, we define ${\bf k}$-bilinear maps 
\begin{align}
\circ_i^\mathrm{Dend} : \mathcal{O}^\mathrm{Dend}(m) \otimes \mathcal{O}^\mathrm{Dend}(n) \rightarrow \mathcal{O}^\mathrm{Dend}(m+n-1) ~~ \text{ by } \nonumber \\
(f \circ_i^\mathrm{Dend} g)^{[r]} := f_{ R_0  (m;1, \ldots, n, \ldots, 1) [r]} \circ_i g_{R_i  (m;1, \ldots, n, \ldots, 1) [r]}, \label{split-pc}
\end{align}
for $f = (f^{[1]}, \ldots, f^{[m]}) \in \mathcal{O}^\mathrm{Dend} (m)$ and $g = (g^{[1]}, \ldots, g^{[n]}) \in \mathcal{O}^\mathrm{Dend} (n)$. Moreover, note that the element $\mathds{1} \in \mathcal{O}(1)$ can be regarded as an element $\mathds{1} = \mathds{1}^{[1]} \in \mathcal{O}^\mathrm{Dend}(1)$. With all these notations, we have the following.

\begin{thm}\label{split-operad-thm}
The triple $\mathcal{O}^\mathrm{Dend} = (\{ \mathrm{O}^\mathrm{Dend} (n) \}_{n \geq 1}, \circ_i^\mathrm{Dend}, \mathds{1})$ is a nonsymmetric operad. Moreover, an element $\pi = (\pi^{[1]}, \pi^{[2]}) \in \mathcal{O}^\mathrm{Dend} (2)$ is a multiplication on the operad $\mathcal{O}^\mathrm{Dend}$ if and only if the pair $(\pi^{[1]}, \pi^{[2]})$ is a dendriform-multiplication on the operad $\mathcal{O}$. 
\end{thm}

\begin{proof}
Let $f = (f^{[1]}, \ldots, f^{[m]}) \in \mathcal{O}^\mathrm{Dend} (m)$, $g = (g^{[1]}, \ldots, g^{[n]}) \in \mathcal{O}^\mathrm{Dend} (n)$ and $h = (h^{[1]}, \ldots, h^{[p]}) \in \mathcal{O}^\mathrm{Dend} (p)$. Let $1 \leq i \leq m$ and $1 \leq j \leq n$. Then for any $[r] \in C_{m+n+p-2}$ with $1 \leq r \leq i-1$, we have
\begin{align*}
((f \circ_i^\mathrm{Dend} g) \circ_{i+j-1}^\mathrm{Dend} h )^{[r]} = (f \circ_i^\mathrm{Dend} g)^{[r]} \circ_{i+j-1} h^{[1] + \cdots + [p]} = ( f^{[r]} \circ_i g^{[1]+\cdots+[n]}) \circ_{i+j-1} h^{[1] + \cdots + [p]}.
\end{align*}
On the other hand,
\begin{align*}
&(f \circ_i^\mathrm{Dend} (g \circ_j^\mathrm{Dend} h))^{[r]} \\
&= f^{[r]} \circ_i (g \circ_j^\mathrm{Dend} h)^{[1] + \cdots + [n+p-1]} \\
&= \sum_{1 \leq s \leq j-1} f^{[r]} \circ_i (g \circ_j^\mathrm{Dend} h)^{[s]} ~+~ \sum_{j \leq s \leq j+p-1} f^{[r]} \circ_i (g \circ_j^\mathrm{Dend} h)^{[s]} ~+~ \sum_{j+p \leq s \leq n+p-1} f^{[r]} \circ_i (g \circ_j^\mathrm{Dend} h)^{[s]} \\
&= \sum_{1 \leq s \leq j-1} f^{[r]} \circ_i (g^{[s]} \circ_j h^{[1] + \cdots + [p]}) ~+~ f^{[r]} \circ_i (g^{[j]} \circ_j h^{[1] + \cdots + [p]}) ~+~ \sum_{j+1 \leq s \leq n} f^{[r]} \circ_i (g^{[s]} \circ_j h^{[1] + \cdots + [p]}) \\
&= f^{[r]} \circ_i (g^{[1] + \cdots + [n]} \circ_j h^{[1] + \cdots + [p]}) \\
&= (f^{[r]} \circ_i g^{[1] + \cdots + [n]}) \circ_{i+j-1} h^{[1] + \cdots + [p]}.
\end{align*}
This shows that $((f \circ_i^\mathrm{Dend} g) \circ_{i+j-1}^\mathrm{Dend} h )^{[r]} = (f \circ_i^\mathrm{Dend} (g \circ_j^\mathrm{Dend} h))^{[r]} $, for $[r] \in C_{m+n+p-2}$ with $1 \leq r \leq i-1$. In a similar way, we can show that this identity holds when $r$ belongs to either of the following intervals: $i \leq r \leq i+j-2$, $i+j-1 \leq r \leq i+j+p-2$, $i+j+p-1 \leq r \leq i+n+p-2$ and $i +n +p-1 \leq r \leq m+n+p-2$. This concludes that
\begin{align*}
(f \circ_i^\mathrm{Dend} g) \circ_{i+j-1}^\mathrm{Dend} h  = f \circ_i^\mathrm{Dend} (g \circ_j^\mathrm{Dend} h), ~ \text{ for } 1 \leq i \leq m \text{ and } 1 \leq j \leq n.
\end{align*}
Similarly, we can show that $(f \circ_i^\mathrm{Dend} g) \circ_{j+n-1}^\mathrm{Dend} h = (f \circ_j^\mathrm{Dend} h) \circ_i^\mathrm{Dend} g$, for $1 \leq i < j \leq m$. Finally, for $f = (f^{[1]}, \ldots, f^{[m]}) \in \mathcal{O}^\mathrm{Dend} (m)$, $1 \leq i \leq m$ and $[r] \in C_m$, we have
\begin{align*}
(f \circ_i^\mathrm{Dend} \mathds{1})^{[r]} = f^{[r]} \circ_i \mathds{1} = f^{[r]}  ~~ \text{ and } ~~ (\mathds{1} \circ_i^\mathrm{Dend} f)^{[r]} = \mathds{1} \circ_1 f^{[r]} = f^{[r]}.
\end{align*}
Therefore, we have $f \circ_i^\mathrm{Dend} \mathds{1} = \mathds{1} \circ_i^\mathrm{Dend} f = f$, for $1 \leq i \leq m$. Thus the triple $(\{ \mathcal{O}^\mathrm{Dend} (n) \}_{n \geq 1} , \circ_i^\mathrm{Dend} , \mathds{1})$ is a nonsymmetric operad.

\medskip

For any $\pi = (\pi^{[1]}, \pi^{[2]}) \in \mathcal{O}^\mathrm{Dend} (2)$, we have $(\pi \circ_1^\mathrm{Dend} \pi - \pi \circ_2^\mathrm{Dend} \pi) \in \mathcal{O}^\mathrm{Dend} (3)$. Moreover, using the partial compositions (\ref{split-pc}), we have
\begin{align*}
(\pi \circ_1^\mathrm{Dend} \pi - \pi \circ_2^\mathrm{Dend} \pi)^{[1]} =~& \pi^{[1]} \circ_1 \pi^{[1]} ~-~ \pi^{[1]} \circ_2 (\pi^{[1]} + \pi^{[2]}),\\
(\pi \circ_1^\mathrm{Dend} \pi - \pi \circ_2^\mathrm{Dend} \pi)^{[2]} =~& \pi^{[1]} \circ_1 \pi^{[2]} ~-~ \pi^{[2]} \circ_2 \pi^{[1]},\\
(\pi \circ_1^\mathrm{Dend} \pi - \pi \circ_2^\mathrm{Dend} \pi)^{[3]} =~& \pi^{[2]} \circ_1 (\pi^{[1]} + \pi^{[2]}) ~-~ \pi^{[2]} \circ_2 \pi^{[2]}.
\end{align*}
This shows that $\pi \circ_1^\mathrm{Dend} \pi - \pi \circ_2^\mathrm{Dend} \pi = 0$ if and only if $({\pi^{[1]}, \pi^{[2]}})$ forms a dendriform-multiplication on the operad $\mathcal{O}$. This completes the proof.
\end{proof}

Let $\mathcal{O}$ be a nonsymmetric operad and $(\pi_\prec, \pi_\succ)$ be a dendriform-multiplication on $\mathcal{O}$. It follows from Theorem \ref{split-operad-thm} that the element $ \pi = (\pi_\prec, \pi_\succ) \in \mathcal{O}^\mathrm{Dend}(2)$ is a multiplication on the operad $\mathcal{O}^\mathrm{Dend}$. Hence there is a cochain complex $\big( C^\bullet_{(\pi_\prec, \pi_\succ)} (\mathcal{O}^\mathrm{Dend}),   \delta_{(\pi_\prec, \pi_\succ)} \big)$ given by
\begin{align*}
C^n_{(\pi_\prec, \pi_\succ)} (\mathcal{O}^\mathrm{Dend}) := \mathcal{O}^\mathrm{Dend}(n) = {\bf k}[C_n] \otimes \mathcal{O}(n), \text{ for } n \geq 1 ~~~ \text{ and } ~~~ \delta_{(\pi_\prec, \pi_\succ)} = \llbracket (\pi_\prec, \pi_\succ), - \rrbracket_{\mathcal{O}^\mathrm{Dend}},
\end{align*}
where $\llbracket ~, ~ \rrbracket_{\mathcal{O}^\mathrm{Dend}}$ denotes the degree $-1$ graded Lie bracket on the graded ${\bf k}$-module $\mathcal{O}^\mathrm{Dend} (\bullet) = \oplus_{n \geq 1} \mathcal{O}^\mathrm{Dend} (n)$ induced by the operad $\mathcal{O}^\mathrm{Dend}$. The corresponding cohomology groups are called the {\bf cohomology induced by the dendriform-multiplication $(\pi_\prec, \pi_\succ)$} and they are denoted by $H^\bullet_{(\pi_\prec, \pi_\succ)} (\mathcal{O}^\mathrm{Dend})$. The graded space of cohomology also carries a Gerstenhaber structure as it is induced by a multiplication on the operad $\mathcal{O}^\mathrm{Dend}.$

\medskip

Let $\mathcal{O}$ be a nonsymmetric operad. Consider the operad $\mathcal{O}^\mathrm{Dend}$ given in Theorem \ref{split-operad-thm}. For each $n \geq 1$, define a ${\bf k}$-linear map $\varphi_n : \mathcal{O}^\mathrm{Dend} (n) \rightarrow \mathcal{O}(n)$ by
\begin{align*}
\varphi_n (f) = f^{[1]} + \cdots + f^{[n]}, \text{ for } f = (f^{[1]}, \ldots, f^{[n]}) \in \mathcal{O}^\mathrm{Dend}(n).
\end{align*}
Then we have the following.

\begin{prop}\label{dend-tot-mor}
The collection $\varphi = \{ \varphi_n : \mathcal{O}^\mathrm{Dend}(n) \rightarrow \mathcal{O}(n) \}_{n \geq 1}$ defines a morphism $\varphi : \mathcal{O}^\mathrm{Dend} \rightarrow \mathcal{O}$ of nonsymmetric operads. In particular, if $(\pi_\prec, \pi_\succ)$ is a dendriform-multiplication on the operad $\mathcal{O}$, then $\varphi$ induces a Gerstenhaber algebra morphism $\varphi_\ast : H^\bullet_{(\pi_\prec, \pi_\succ)} (\mathcal{O}^\mathrm{Dend}) \rightarrow H^\bullet_{\pi_{\mathrm{Tot}}} (\mathcal{O})$ between cohomology groups, where $\pi_\mathrm{Tot}  = \pi_\prec + \pi_\succ$ is the total multiplication.
\end{prop}

\begin{proof}
Let $f = (f^{[1]}, \ldots, f^{[m]}) \in \mathcal{O}^\mathrm{Dend} (m)$, $g = (g^{[1]}, \ldots, g^{[n]}) \in \mathcal{O}^\mathrm{Dend} (n)$ and $1 \leq i \leq m$. Then
\begin{align*}
\varphi_{m+n-1} (f \circ_i^\mathrm{Dend} g) =~& (f \circ_i^\mathrm{Dend} g)^{[1]} + \cdots + (f \circ_i^\mathrm{Dend} g)^{[m+n-1]} \\
=~& \sum_{1 \leq r \leq i-1} (f \circ_i^\mathrm{Dend} g)^{[r]} + \sum_{i \leq r \leq i+n-1} (f \circ_i^\mathrm{Dend} g)^{[r]} + \sum_{i+n \leq r \leq m+n-1} (f \circ_i^\mathrm{Dend} g)^{[r]} \\
=~& \sum_{1 \leq r \leq i-1} f^{[r]} \circ_i g^{[1] + \cdots + [n]} ~+~ f^{[i]} \circ_i g^{[1] + \cdots + [n]} ~+~ \sum_{i+1 \leq r \leq m} f^{[r]} \circ_i g^{[1] + \cdots + [n]} \\
=~& f^{[1] + \cdots + [m]} \circ_i g^{[1] + \cdots + [n]} = \varphi_m (f) \circ_i \varphi_n (g).
\end{align*}
Moreover, we have $\varphi_1 (\mathds{1}) = \mathds{1}$. Hence $\varphi$ is a morphism of nonsymmetric operads.

The last part follows as $\varphi_2 ((\pi_\prec, \pi_\succ)) = \pi_\prec + \pi_\succ = \pi_\mathrm{Tot}.$
\end{proof}

\begin{remark}
Let $\mathcal{O}$ be the endomorphism operad $\mathrm{End}_A$ associated to the ${\bf k}$-module $A$. Let $(\pi_\prec, \pi_\succ)$ be a dendriform-multiplication on the operad $\mathrm{End}_A$ which corresponds to the dendriform algebra structure $(A, \prec, \succ).$ Then it is easy to see that the cochain complex induced by the dendriform-multiplication $(\pi_\prec, \pi_\succ)$ coincides with the standard cochain complex for the dendriform algebra $(A, \prec, \succ)$ given in \cite{lod-val-book,A4,yau}. In this sense, the cohomology induced by dendriform-multiplication generalizes the cohomology of a dendriform algebra. Moreover, it follows from Proposition \ref{dend-tot-mor} that there is a morphism from the cohomology of a dendriform algebra to the Hochschild cohomology of the total associative algebra.
\end{remark}

\subsection{Other types of split multiplications}

In Subsection \ref{subsec-dend-mul}, we defined dendriform-multiplication on a nonsymmetric operad $\mathcal{O}$. It is given by a pair of two elements of $\mathcal{O}(2)$ ~(neither of them is a multiplication on $\mathcal{O}$) satisfying some set of relations which imply that their sum is a multiplication. Thus, dendriform-multiplications on $\mathcal{O}$ can be regarded as a splitting of multiplications on $\mathcal{O}$ in the same way dendriform algebras are splitting of associative algebras. Besides dendriform algebras, J.-L. Loday and his collaborators introduced a few other algebras (e.g. diassociative, triassociative, tridendriform, quadri algebras etc.) in their study of planar trees, shuffles and other combinatorial objects \cite{loday,loday-ronco,loday-aguiar}. All these algebras are Loday-algebras. Moreover, these algebras are splitting of associative algebras in the sense that the sum of the defining operations of these algebras forms associative product. 

Let `Lod' denotes a fixed type among these Loday-algebras  (e.g. diassociative, triassociative, dendriform, tridendriform, quadri etc.). Then it has been observed in \cite{das,yau} that there is a collection of nonempty sets $\{ U_n \}_{n \geq 1}$ and some `nice' combinatorial maps 
\begin{align*}
R_0 (m;1, \ldots, n, \ldots, 1) : U_{m+n-1} \rightarrow U_m ~~ \text{ and }~~ R_i (m;1, \ldots, n, \ldots, 1) : U_{m+n-1} \rightarrow {\bf k}[C_n],
\end{align*}
for $m, n \geq 1$ and $1 \leq i \leq m$ which makes the collection of {\bf k}-modules
\begin{align}\label{end-rep}
\mathrm{End}_A^\mathrm{Lod} (n) := {\bf k}[U_n] \otimes \mathrm{End}_A (n), \text{ for } n \geq 1
\end{align}
with partial compositions
\begin{align*}
(f \circ_i^\mathrm{Lod} g)^{[r]} = f^{R_0 (m;1, \ldots, n, \ldots, 1)[r]} \circ_i g^{R_i (m;1, \ldots, n, \ldots, 1)[r]}, \text{ for } f \in \mathrm{End}_A^\mathrm{Lod} (m),~ g \in \mathrm{End}_A^\mathrm{Lod} (n)
\end{align*}
into a nonsymmetric operad. For the above Loday-algebras, the corresponding sets $\{ U_n \}_{n \geq 1}$ are listed below:

\medskip

\begin{center}
\begin{tabular}{l|l}
Types of algebras & $U_n$ \\ \hline 
Diassociative algebras & $PBT_n$ (set of planar binary trees with $n$ vertices) \\
Triassociative algebras & $PB_n$ (set of planar trees with $n$ vertices) \\
Dendriform algebras & $C_n$ \\
Tridendriform algebras & $P_n$ (the set of nonempty subsets of $\{1, \ldots, n \}$) \\
Quadri algebras & $C_n \times C_n$.
\end{tabular}
\end{center}

\medskip

\noindent The combinatorial maps $\{ R_0, R_i \}$ for these Loday-algebras are explicitly defined in \cite{das} (In the case of dendriform algebras, these maps are given in (\ref{r0}), (\ref{ri})). A Lod-algebra structure on $A$ is equivalent to multiplication on the operad $\mathrm{End}_A^\mathrm{Lod}$. This idea can be easily generalized by replacing the endomorphism operad $\mathrm{End}_A$ in (\ref{end-rep}) by any nonsymmetric operad.

Let $\mathcal{O}$ be a nonsyymetric operad. For any fixed type `Lod' of the above Loday-algebras, we define a nonsymmetric operad $\mathcal{O}^\mathrm{Lod} = (\{ \mathcal{O}^\mathrm{Lod} (n) \}_{n \geq 1}, \circ_i^\mathrm{Lod}, \mathds{1})$ as follows
\begin{align*}
\mathcal{O}^\mathrm{Lod} (n) = {\bf k}[U_n] \otimes \mathcal{O}(n) ~~~ \text{ and } ~~~ (f\circ_i^\mathrm{Lod} g)^{[r]} = f^{R_0 (m;1, \ldots, n, \ldots, 1)[r]} \circ_i g^{R_i (m;1, \ldots, n, \ldots, 1)[r]}.
\end{align*}

\begin{defn}
A {\bf Lod-multiplication} on $\mathcal{O}$ is a multiplication on the operad $\mathcal{O}^\mathrm{Lod}$.
\end{defn}

When $\mathcal{O} = \mathrm{Dend}$, the type of dendriform algebras, the operad $\mathcal{O}^\mathrm{Dend}$ coincides with the one given in the previous subsection. Hence a Dend-multiplication on $\mathcal{O}$ is simply a dendriform-multiplication on $\mathcal{O}$ defined in Definition \ref{defn-dend-mul}.

By writing down the descriptions of the operad $\mathcal{O}^\mathrm{Lod}$ for various types `Lod' of Loday-algebras, one can explicitly define Lod-multiplications.  Following the sets $\{ U_n \}_{n \geq 1}$ and combinatorial maps $\{ R_0, R_i \}$ for Lod = TriDend, the type of tridendriform algebras, one gets the operad $\mathcal{O}^\mathrm{TriDend}$. A multiplication on the operad $\mathcal{O}^\mathrm{TriDend}$ is called a tridendriform-multiplication on $\mathcal{O}$. Explicitly, it is given by the following.

\begin{defn}
A {\bf tridendriform-multiplication} on $\mathcal{O}$ consists of a triple $(\pi_\prec, \pi_\succ, \pi_\odot)$ of elements of $\mathcal{O}(2)$ satisfying the following set of identities
\begin{align}
\pi_\prec \circ_1 \pi_\prec =~& \pi_\prec \circ_2 (\pi_\prec + \pi_\succ + \pi_\odot), \label{tridend-1}\\
\pi_\prec \circ \pi_\succ =~& \pi_\succ \circ_2 \pi_\prec,\\
\pi_\succ \circ_1 (\pi_\prec + \pi_\succ + \pi_\odot) =~& \pi_\succ \circ_2 \pi_\succ,\\
\pi_\prec \circ_1 \pi_\odot =~& \pi_\odot \circ_2 \pi_\prec,\\
\pi_\odot \circ_1 \pi_\prec =~& \pi_\odot \circ_2 \pi_\succ,\\
\pi_\odot \circ_1 \pi_\succ =~& \pi_\succ \circ_2 \pi_\odot,\\
\pi_\odot \circ_1 \pi_\odot =~& \pi_\odot \circ_2 \pi_\odot. \label{tridend-7}
\end{align}
\end{defn}

When $\mathcal{O}$ is the endomorphism operad $\mathrm{End}_A$, a tridendriform-multiplication on $\mathrm{End}_A$ is equivalent to a tridendriform algebra structure on $A$. In the following, we observe that a tridendriform-multiplication splits multiplication.

\begin{prop}
Let $(\pi_\prec, \pi_\succ, \pi_\odot)$ be a tridendriform-multiplication on the nonsymmetric operad $\mathcal{O}$. Then $\pi_\prec + \pi_\succ + \pi_\odot$ is a multiplication on $\mathcal{O}$.
\end{prop}

\begin{proof}
By adding the left-hand sides of the identities (\ref{tridend-1})-(\ref{tridend-7}), we simply get $(\pi_\prec + \pi_\succ + \pi_\odot) \circ_1 (\pi_\prec + \pi_\succ + \pi_\odot)$. On the other hand, by adding the right-hand sides, we get $(\pi_\prec + \pi_\succ + \pi_\odot) \circ_2 (\pi_\prec + \pi_\succ + \pi_\odot)$. This shows that $\pi_\prec + \pi_\succ + \pi_\odot$ is a multiplication on $\mathcal{O}$.
\end{proof}

Note that a tridendriform-multiplication $(\pi_\prec, \pi_\succ, \pi_\odot)$ on the nonsymmetric operad $\mathcal{O}$ with $\pi_\odot = 0$ is nothing but a dendriform-multiplication. In general, a tridendriform-multiplication always induces a dendriform-multiplication. The proof of the following result is straightforward. Hence we omit the proof.

\begin{prop}
Let $(\pi_\prec, \pi_\succ, \pi_\odot)$ be a tridendriform-multiplication on the nonsymmetric operad $\mathcal{O}$. Then $(\pi_\prec + \pi_\odot, \pi_\succ )$ is a dendriform-multiplication on $\mathcal{O}$.
\end{prop}

%\subsection{Cohomology induced by dendriform-multiplication}

\section{Dendriform-family algebras}\label{sec-5}

Let $\Omega$ be a fixed semigroup. Given a nonsymmetric operad $\mathcal{O}$, here we first construct a new nonsymmetric operad $\mathrm{Fam}(\mathcal{O}^\Omega)^\mathrm{Dend}$, When $\mathcal{O} = \mathrm{End}_A$ the endomorphism operad associated to $A$, a multiplication on the operad $\mathrm{Fam}(\mathrm{End}_A^\Omega)^\mathrm{Dend}$ corresponds to a dendriform-family structure on $A$. This characterization allows us to define the cohomology of a dendriform-family algebra. Finally, we introduce and study homotopy dendriform-family algebras ($Dend_\infty$-family algebras in short).

%\textcolor{red}{What is $\Omega$?}

\subsection{Dendriform-family algebras and the operad $\mathrm{Fam}(\mathcal{O}^\Omega)^\mathrm{Dend}$}

\begin{defn} \cite{zhang-free}
A {\bf dendriform-family algebra} is a pair $(A, \{ \prec_\alpha, \succ_\alpha \}_{\alpha \in \Omega})$ consisting of a ${\bf k}$-module $A$ together with a collection $\{ \prec_\alpha, \succ_\alpha : A \otimes A \rightarrow A \}_{\alpha \in \Omega}$ of ${\bf k}$-bilinear maps on $A$ satisfying
\begin{align*}
(a \prec_\alpha b) \prec_\beta c                              &= a \prec_{\alpha \beta} (b \prec_\beta c + b \succ_\alpha c)\\
(a \succ_\alpha b) \prec_\beta c                              &= a \succ_\alpha (b \prec_\beta c),\\
(a \prec_\beta b + a \succ_\alpha b) \succ_{\alpha \beta} c  &= a \succ_\alpha (b \succ_\beta c), \text{ for } a,b,c \in A \text{ and } \alpha, \beta \in \Omega.
\end{align*}
\end{defn}

One can construct a dendriform algebra from a dendriform-family algebra \cite{zhang}. 

\begin{prop}\label{dend-fam-dend}
Let $(A, \{ \prec_\alpha, \succ_\alpha \}_{\alpha \in \Omega})$ be a dendriform-family algebra. Then the {\bf k}-module $A \otimes {\bf k} \Omega$ equipped with the operations $\prec, \succ : (A \otimes {\bf k} \Omega) \otimes (A \otimes {\bf k} \Omega) \rightarrow (A \otimes {\bf k} \Omega)$ is a dendriform algebra, where
\begin{align*}
(a \otimes \alpha ) \prec (b \otimes \beta) = a \prec_\beta b \otimes \alpha \beta  ~ \text{ and } ~ (a \otimes \alpha ) \succ (b \otimes \beta) = a \succ_\alpha b \otimes \alpha \beta.
\end{align*}
\end{prop}

In \cite{fard-bondia} the authors introduced a notion of Rota-Baxter family as a generalization of Rota-Baxter operator. Let $A$ be an associative algebra (with the product $\cdot$ ). A {\bf Rota-Baxter family} on the algebra $A$ is a collection $\{ R_\alpha : A \rightarrow A \}_{\alpha \in \Omega}$ of ${\bf k}$-linear maps satisfying
\begin{align*}
R_\alpha (a) \cdot R_\beta (b) = R_{\alpha \beta} (R_\alpha (a) \cdot b + a \cdot R_\beta (b)), \text{ for } a, b \in A \text{ and } \alpha, \beta \in \Omega.
\end{align*}

%A Rota-Baxter family on $A$ induces a dendriform-family algebra \cite{zhang-free}. 
Let $\{ R_\alpha : A \rightarrow A \}_{\alpha \in \Omega}$ be a Rota-Baxter family on the algebra $A$. Then it has been shown in \cite{zhang-free} that the pair $(A, \{ \prec_\alpha, \succ_\alpha \}_{\alpha \in \Omega})$ is a dendriform-family algebra, where
 \begin{align*}
 a \prec_\alpha b = a \cdot R_\alpha (b) ~~ \text{ and } ~~ a \succ_\alpha b = R_\alpha (a) \cdot b,~ \text{ for } a, b \in A, \alpha \in \Omega.
 \end{align*}

The following notion of associative algebra relative to semigroup was introduced by Aguiar \cite{agu} which are related to dendriform-family algebras.

\begin{defn}
An {\bf associative algebra relative to $\Omega$} is a ${\bf k}$-module $A$ equipped with a collection of ${\bf k}$-bilinear maps $\{ \cdot_{\alpha, \beta} : A \otimes A \rightarrow A \}_{\alpha, \beta \in \Omega}$ satisfying
\begin{align}\label{ass-rel}
(a \cdot_{\alpha, \beta} b) \cdot_{\alpha \beta, \gamma} c = a \cdot_{\alpha, \beta \gamma} (b \cdot_{\beta, \gamma} c),~\text{ for } a, b, c \in A \text{ and } \alpha, \beta, \gamma \in \Omega.
\end{align}
\end{defn}

\begin{prop}
If $(A, \{ \prec_\alpha , \succ_\alpha \}_{\alpha \in \Omega})$ is a dendriform-family algebra, then $(A, \{ \cdot_{\alpha, \beta} : A \otimes A \rightarrow A \}_{\alpha , \beta \in \Omega})$ is an associative algebra relative to $\Omega$, where $a \cdot_{\alpha, \beta} b := a \prec_{\beta} b + a \succ_\alpha b$, for $a, b \in A$.
\end{prop}

%For any graded vector space $\mathcal{A} = \bigoplus_{i \in \mathbb{Z}} \mathcal{A}^i$, we consider the space $\mathrm{Hom}_{\Omega^{\times k}} (\mathcal{A}^{\otimes k}, \mathcal{A})$.

Let $\mathcal{O}$ be a nonsymmetric operad. 
We will now construct a nonsymmetric operad $\mathcal{O}^\Omega$ which will helps us to construct our required operad in the end. For each $n \geq 1$, we consider the ${\bf k}$-module
\begin{align*}
\mathcal{O}^\Omega (n) = {\bf k}[\Omega^{\times n}] \otimes \Omega (n).
\end{align*}
An element $f \in \mathcal{O}^\Omega (n)$ is a collection $f = \{ f_{\alpha_1, \ldots, \alpha_n} \in \mathcal{O}(n)  \}_{\alpha_1, \ldots, \alpha_n \in \Omega}$ of elements of $\mathcal{O}(n)$ labelled by $\Omega^{\times n}$. Note that the element $\mathds{1} \in \mathcal{O}(1)$ can be regarded as an element $\mathds{1} = \{ \mathds{1}_\alpha = \mathds{1} \in \mathcal{O}(1) \}_{\alpha \in \Omega} \in \mathcal{O}^\Omega (1)$. For each $m, n \geq 1$ and $1 \leq i \leq m$, we define a ${\bf k}$-bilinear map $\circ_i^\Omega : \mathcal{O}^\Omega (m) \otimes \mathcal{O}^\Omega (n) \rightarrow \mathcal{O}^\Omega (m+n-1)$ by
\begin{align}\label{omega-op}
(f \circ_i^\Omega g)_{\alpha_1, \ldots, \alpha_{m+n-1}} := f_{\alpha_1, \ldots, \alpha_i \cdots \alpha_{i+n-1}, \ldots, \alpha_{m+n-1}} \circ_i g_{\alpha_i, \ldots, \alpha_{i+n-1}},
\end{align}
for $f \in \mathcal{O}^\Omega (m)$, $g \in \mathcal{O}^\Omega (n)$ and $\alpha_1, \ldots, \alpha_{m+n-1} \in \Omega$.

\begin{thm}
The triple $\mathcal{O}^\Omega = (  \{ \mathcal{O}^\Omega (n) \}_{n \geq 1}, \circ_i^\Omega, \mathds{1} )$ is a nonsymmetric operad.
\end{thm}

\begin{proof}
Let $f = \{ f_{\alpha_1, \ldots, \alpha_m} \in \mathcal{O}(m) \}_{\alpha_1, \ldots, \alpha_m \in \Omega} \in \mathcal{O}^\Omega (m)$, $g = \{ g_{\alpha_1, \ldots, \alpha_n} \in \mathcal{O}(n) \}_{\alpha_1, \ldots, \alpha_n \in \Omega} \in \mathcal{O}^\Omega (n)$ and $h = \{ h_{\alpha_1, \ldots, \alpha_p} \in \mathcal{O}(p) \}_{\alpha_1, \ldots, \alpha_p \in \Omega} \in \mathcal{O}^\Omega (p)$. Then for $1 \leq i \leq m$ and $1 \leq j \leq n$, we have
\begin{align*}
&(( f \circ_i^\Omega g) \circ_{i+j-1}^\Omega h)_{\alpha_1, \ldots, \alpha_{m+n+p-2}} \\
&= ( f \circ_i^\Omega g)_{\alpha_1, \ldots, \alpha_{i+j-1} \cdots \alpha_{i+j+p-2}, \ldots, \alpha_{m+n+p-2}} \circ_{i+j-1} h_{\alpha_{i+j-1} , \ldots, \alpha_{i+j+p-2}} \\
&= \big(  f_{\alpha_1, \ldots, \alpha_i \cdots \alpha_{i+n+p-2}, \ldots, \alpha_{m+n+p-2}} \circ_i g_{\alpha_i, \ldots, \alpha_{i+j-1} \cdots \alpha_{i+j+p-2}, \ldots, \alpha_{i+n+p-2}}   \big) \circ_{i+j-1} h_{\alpha_{i+j-1} , \ldots, \alpha_{i+j+p-2}} \\
&=f_{\alpha_1, \ldots, \alpha_i \cdots \alpha_{i+n+p-2}, \ldots, \alpha_{m+n+p-2}} \circ_i \big(   g_{\alpha_i, \ldots, \alpha_{i+j-1} \cdots \alpha_{i+j+p-2}, \ldots, \alpha_{i+n+p-2}} \circ_j    h_{\alpha_{i+j-1} , \ldots, \alpha_{i+j+p-2}}  \big) \\
&= f_{\alpha_1, \ldots, \alpha_i \cdots \alpha_{i+n+p-2}, \ldots, \alpha_{m+n+p-2}} \circ_i (g \circ_j^\Omega h)_{\alpha_i, \ldots, \alpha_{i+n+p-2}} \\
&= (f \circ_i^\Omega (g \circ_j^\Omega h))_{\alpha_1, \ldots, \alpha_{m+n+p-2}}.
\end{align*}
Similarly, for $1 \leq i < j \leq m$,
\begin{align*}
&((f \circ_i^\Omega g) \circ_{j+n-1}^\Omega h )_{\alpha_1, \ldots, \alpha_{m+n+p-2}} \\
&= (f \circ_i^\Omega g)_{\alpha_1, \ldots, \alpha_{j+n-1} \cdots \alpha_{j+n+p-2}, \ldots, \alpha_{m+n+p-2}} \circ_{j+n-1} h_{\alpha_{j+n-1}, \ldots, \alpha_{j+n+p-2}} \\
&= \big( f_{\alpha_1, \ldots, \alpha_i \cdots \alpha_{i+n-1}, \ldots, \alpha_{j+n-1} \cdots \alpha_{j+n+p-2}, \ldots, \alpha_{m+n+p-2}}  \circ_i g_{\alpha_i, \ldots, \alpha_{i+n-1}} \big) \circ_{j+n-1} h_{\alpha_{j+n-1}, \ldots, \alpha_{j+n+p-2}} \\
&= \big( f_{\alpha_1, \ldots, \alpha_i \cdots \alpha_{i+n-1}, \ldots, \alpha_{j+n-1} \cdots \alpha_{j+n+p-2}, \ldots, \alpha_{m+n+p-2}} \circ_j  h_{\alpha_{j+n-1}, \ldots, \alpha_{j+n+p-2}}  \big) \circ_i g_{\alpha_i, \ldots, \alpha_{i+n-1}} \\
&= (f \circ_j h)_{\alpha_1, \ldots, \alpha_i \cdots \alpha_{i+n-1}, \ldots, \alpha_{m+n+p-2}} \circ_i g_{\alpha_i, \ldots, \alpha_{i+n-1}}\\
&= ((f \circ_j h) \circ_i g)_{\alpha_1, \ldots, \alpha_{m+n+p-2}}.
\end{align*}
Moreover, for $f = \{ f_{\alpha_1, \ldots, \alpha_m} \in \mathcal{O}(m) \}_{\alpha_1, \ldots, \alpha_m \in \Omega} \in \mathcal{O}^\Omega (m)$ and $1 \leq i \leq m$, we have
\begin{align*}
&( f \circ_i^\Omega \mathds{1})_{\alpha_1, \ldots, \alpha_m} = f_{\alpha_1, \ldots, \alpha_m} \circ_i \mathds{1}_{\alpha_i} = f_{\alpha_1, \ldots, \alpha_m} \circ_i \mathds{1} = f_{\alpha_1, \ldots, \alpha_m}, \\
&(\mathds{1} \circ_1^\Omega f)_{\alpha_1, \ldots, \alpha_m} = \mathds{1}_{\alpha_1 \cdots \alpha_m} \circ_1 f_{\alpha_1, \ldots, \alpha_m} = \mathds{1} \circ_1 f_{\alpha_1, \ldots, \alpha_m} = f_{\alpha_1, \ldots, \alpha_m}.
\end{align*}
This shows that $f \circ_i^\Omega \mathds{1} = \mathds{1} \circ_1 f = f$. Therefore, the triple $\mathcal{O}^\Omega = (  \{ \mathcal{O}^\Omega (n) \}_{n \geq 1}, \circ_i^\Omega, \mathds{1} )$ is a nonsymmetric operad.
\end{proof}

Let $\mathcal{O} = \mathrm{End}_A$ be the endomorphism operad associated to a ${\bf k}$-module $A$. Then the operad $\mathcal{O}^\Omega = \mathrm{End}_A^\Omega$ is given by $\mathrm{End}_A^\Omega (n) = {\bf k}[\Omega^{\times n}] \otimes \mathrm{End}_A (n), \text{ for } n \geq 1$, and the partial compositions are given by (\ref{omega-op}) where on the right hand side we use the partial compositions of the endomorphism operad $\mathrm{End}_A$.

An element $\pi \in \mathrm{End}_A^\Omega (2)$ is equivalent to having a collection $\{ ~\cdot_{\alpha, \beta} :  A \otimes A \rightarrow A \}_{\alpha, \beta \in \Omega}$ of ${\bf k}$-bilinear maps given by $a \cdot_{\alpha, \beta} b = \pi_{\alpha, \beta} (a, b)$, for $a, b \in A$. Note that $\pi$ is a multiplication on the operad $\mathrm{End}_A^\Omega$ if and only if $\pi \circ_1^\Omega \pi = \pi \circ_2^\Omega \pi$, equivalently, the identity (\ref{ass-rel}) holds. 
In other words, $\pi$ defines a structure of associative algebra relative to $\Omega$ on the ${\bf k}$-module $A$. As a summary, we conclude that if $\mathcal{O}$ is the operad to study `Lod' algebras in terms of multiplications, then $\mathcal{O}^\Omega$ is the operad to study Lod-family algebras. Here `Lod' is any fixed type of Loday-algebras.

\medskip

Given any nonsymmetric operad $\mathcal{O}$, in the previous Section, we constructed an operad $\mathcal{O}^\mathrm{Dend}$. Since $\mathcal{O}^\Omega$ is a nonsymmetric operad, it follows that $(\mathcal{O}^\Omega)^\mathrm{Dend}$ is a nonsymmetric operad, where
\begin{align*}
(\mathcal{O}^\Omega)^\mathrm{Dend} (n) := {\bf k}[C_n] \otimes \mathcal{O}^\Omega (n) = {\bf k}[C_n] \otimes {\bf k}[\Omega^{\times n}] \otimes \mathcal{O}(n), \text{ for } n \geq 1.
\end{align*}
An elements of $(\mathcal{O}^\Omega)^\mathrm{Dend} (n)$ is an $n$-tuple $f = (f^{[1]}, \ldots, f^{[n]}),$ where each $f^{[r]} \in \mathcal{O}^\Omega (n)$, for $1 \leq r \leq n$. That is,
\begin{align*}
f^{[r]} = \big\{ f^{[r]}_{\alpha_1, \ldots, \alpha_n} \in \mathcal{O}(n)  \big\}_{\alpha_1, \ldots, \alpha_n \in \mathcal{O}(n)}, \text{ for } 1 \leq r \leq n.
\end{align*}
One can easily write the partial compositions of the operad $(\mathcal{O}^\Omega)^\mathrm{Dend}$. Let $\llbracket ~, ~ \rrbracket_{(\mathcal{O}^\Omega)^\mathrm{Dend}}$ denotes the degree $-1$ graded Lie bracket on the graded ${\bf k}$-module $(\mathcal{O}^\Omega)^\mathrm{Dend} (\bullet)$.

Consider the ${\bf k}$-submodule~ $\mathrm{Fam}(\mathcal{O}^\Omega)^\mathrm{Dend} (n) \subset (\mathcal{O}^\Omega)^\mathrm{Dend} (n)$ given by
\begin{align*}
\mathrm{Fam}(\mathcal{O}^\Omega)^\mathrm{Dend} (n) = \big\{   f = (f^{[1]}, \ldots, f^{[n]}) \in (\mathcal{O}^\Omega)^\mathrm{Dend} (n) ~ \big| ~ f^{[r]}_{\alpha_1, \ldots, \alpha_n} \text{ does not depend on } \alpha_r, \text{ for } 1 \leq r \leq n \big\}.
\end{align*}
It is easy to verify that the collection of ${\bf k}$-modules $\big\{ \mathrm{Fam}(\mathcal{O}^\Omega)^\mathrm{Dend} (n) \big\}_{n \geq 1}$ is closed under the partial compositions. Hence $\mathrm{Fam}(\mathcal{O}^\Omega)^\mathrm{Dend}$ is a nonsymmetric operad (which is a suboperad of $(\mathcal{O}^\Omega)^\mathrm{Dend}$). An element $\pi \in \mathrm{Fam}(\mathcal{O}^\Omega)^\mathrm{Dend} (2)$ is a pair $\pi = (\pi^{[1]}, \pi^{[2]})$, where
\begin{align*}
\pi^{[1]} =~& \{ \pi^{[1]}_{\alpha , \beta } \in \mathcal{O}(2) ~|~ \pi^{[1]}_{\alpha, \beta} \text{does not depend on } \alpha \}_{\alpha, \beta \in \Omega},\\
\pi^{[2]} =~& \{ \pi^{[2]}_{\alpha , \beta } \in \mathcal{O}(2) ~|~ \pi^{[2]}_{\alpha, \beta} \text{does not depend on } \beta \}_{\alpha, \beta \in \Omega}.
\end{align*}

\begin{remark}
Let $\mathcal{O} = \mathrm{End}_A$ be the endomorphism operad associated to the ${\bf k}$-module $A$. Any element $\pi \in \mathrm{Fam}(\mathrm{End}_A^\Omega)^\mathrm{Dend} (2)$ corresponds to a collection $\{ \prec_\alpha, \succ_\alpha : A \otimes A \rightarrow A \}_{\alpha, \beta \in \Omega}$ of ${\bf k}$-bilinear maps on $A$ by
\begin{align*}
a \prec_\alpha b = \pi^{[1]}_{-, \alpha} (a, b) ~~~ \text{ and } ~~~  a \succ_\alpha b = \pi^{[2]}_{\alpha, - } (a, b), ~ \text{ for } a, b \in A.
\end{align*}
In the above, the notation $`-'$ indicates that the term does not depend on $`-'$. It is easy to see that $\pi$ is a multiplication on the operad $\mathrm{Fam}(\mathrm{End}_A^\Omega)^\mathrm{Dend}$ if and only if the collection $\{ \prec_\alpha, \succ_\alpha \}_{\alpha \in \Omega}$ defines a dendriform-family structure on the ${\bf k}$-module $A$.
\end{remark}

Let $(A, \{ \prec_\alpha, \succ_\alpha \}_{\alpha \in \Omega} )$ be a dendriform-family algebra. Consider the multiplication $\pi \in \mathrm{Fam}(\mathrm{End}_A^\Omega)^\mathrm{Dend} (2)$ associated to the dendriform-family structure. Since $\pi$ is a multiplication, it induces a cochain complex $\big(  C^\bullet_\pi ( \mathrm{Fam}(\mathrm{End}_A^\Omega)^\mathrm{Dend}   ), \delta_\pi \big)$, where
\begin{align*}
C^n_\pi (\mathrm{Fam}(\mathrm{End}_A^\Omega)^\mathrm{Dend}) = \mathrm{Fam}(\mathrm{End}_A^\Omega)^\mathrm{Dend}(n) ~~ \text{ and } ~~ \delta_\pi = \llbracket \pi, - \rrbracket_{ (\mathrm{End}_A^\Omega)^\mathrm{Dend}}.
\end{align*}
The corresponding cohomology groups are called the {\bf cohomology of the dendriform-family algebra} $(A, \{ \prec_\alpha, \succ_\alpha \}_{\alpha \in \Omega}).$ Since the cohomology is induced by multiplication on a nonsymmetric operad, the cohomology inherits a Gerstenhaber algebra structure. This generalizes the Gerstenhaber structure on the cohomology of an ordinary dendriform algebra \cite{A4,yau}.

\subsection{Homotopy dendriform-family algebras}
In this subsection, we first introduce $A_\infty$-algebras relative to $\Omega$ as the homotopy analogue of associative algebras relative to $\Omega$. Then we introduce homotopy dendriform-family algebras ($Dend_\infty$-family algebras) and study various results regarding such homotopy algebras. 

Let $\mathcal{A} = \oplus_{i \in \mathbb{Z}} \mathcal{A}^i$ be a graded ${\bf k}$-module. For any $k \geq 1$, we consider the space $\mathrm{Hom}_{\Omega^{\times k}} (\mathcal{A}^{\otimes k} , \mathcal{A})$. An element $\mu^k \in \mathrm{Hom}_{\Omega^{\times k}} (\mathcal{A}^{\otimes k} , \mathcal{A})$ is a collection
\begin{align*}
\mu^k = \{ \mu^k_{\alpha_1, \ldots, \alpha_k} : \mathcal{A}^{\otimes k} \rightarrow \mathcal{A} \}_{\alpha_1, \ldots, \alpha_k \in \Omega} \text{ of } k\text{-ary multilinear maps labelled by } \Omega^{\times k}.
\end{align*}

\begin{defn}
An {\bf $A_\infty$-algebra relative to $\Omega$} is a graded ${\bf k}$-module $\mathcal{A} = \oplus_{i \in \mathbb{Z}} \mathcal{A}^i$ together with a collection $\{ \mu^1, \mu^2, \ldots \}$, where
\begin{align*}
 \mu^k \in \mathrm{Hom}_{\Omega^{\times k}} (\mathcal{A}^{\otimes k} , \mathcal{A}) \text{ with } \mathrm{deg } (\mu^k_{\alpha_1, \ldots, \alpha_k}) = k-2, \text{ for all } \alpha_1, \ldots, \alpha_k \text{ and } k \geq 1
\end{align*}
satisfying the following set of identities
\begin{align}\label{inf-rel-om}
\sum_{m+n=N+1} \sum_{i=1}^m &(-1)^{i (n+1) + n (|a_1| + \cdots + |a_{i-1}|)}~ \\
& \mu^m_{\alpha_1, \ldots, \alpha_i \cdots \alpha_{i+n-1}, \ldots, \alpha_{N}} \big(    a_1, \ldots, a_{i-1}, \mu^n_{\alpha_i, \ldots, \alpha_{i+n-1}} (a_i, \ldots a_{i+n-1}), a_{i+n}, \ldots, a_{N} \big) = 0, \nonumber
\end{align}
for homogeneous elements $a_1, \ldots, a_N \in \mathcal{A}$; $\alpha_1, \ldots, \alpha_N \in \Omega$ and $N \geq 1$.
\end{defn}

In the rest of the paper, we will simply write $\pm$ for the sign $(-1)^{i (n+1) + n (|a_1| + \cdots + |a_{i-1}|)}$. Note that $\pm$ depents on $m, n , i$ and the degrees of the elements $a_1, \ldots, a_i$. When $\Omega$ is singleton, we recover the classical notion of $A_\infty$-algebras \cite{stas}. 

Let $(\mathcal{A}, \{ \mu^1, \mu^2, \ldots \})$ be an $A_\infty$-algebra relative to $\Omega$ and the underlying graded ${\bf k}$-module $\mathcal{A}$ is concentrated in arity $0$. In other words, $\mathcal{A} = \mathcal{A}^0$ is a (non-graded) ${\bf k}$-module. Then it follows from the degree reason that $\mu^k = 0$, for $k \neq 2$. The element $\mu^2$ gives rise to a collection of ${\bf k}$-bilinear maps
\begin{align*}
\{ ~\cdot_{\alpha, \beta} : \mathcal{A}^0 \otimes \mathcal{A}^0 \rightarrow \mathcal{A}^0 \}_{\alpha, \beta \in \Omega}, \text{ where } a \cdot_{\alpha, \beta} b = \mu^2_{\alpha, \beta} (a, b), \text{ for } a, b \in \mathcal{A}^0 \text{ and } \alpha, \beta \in \Omega.
\end{align*}
Finally, the identities (\ref{inf-rel-om}) is equivalent to the fact that $\{ ~\cdot_{\alpha, \beta} : \mathcal{A}^0 \otimes \mathcal{A}^0 \rightarrow \mathcal{A}^0 \}_{\alpha, \beta \in \Omega}$ defines an associative algebra relative to $\Omega$ on the ${\bf k}$-module $\mathcal{A}^0$. Therefore, $A_\infty$-algebras relative to $\Omega$ can be regarded as a homotopy generalization of associative algebras relative to $\Omega$.

For the graded ${\bf k}$-module $\mathcal{A}$, we consider another space $\mathrm{Hom}_{\Omega^{\times k}} ({\bf k}[C_k] \otimes \mathcal{A}^{\otimes k} , \mathcal{A})$. An element of $\mathrm{Hom}_{\Omega^{\times k}} ({\bf k}[C_k] \otimes \mathcal{A}^{\otimes k} , \mathcal{A})$ is a $k$-tuple $\eta^k = (\eta^{k,[1]}, \ldots, \eta^{k, [k]})$, where each $\eta^{k, [r]} \in \mathrm{Hom}_{\Omega^{\times k}} ( \mathcal{A}^{\otimes k} , \mathcal{A})$. There is a subspace $\mathrm{FamHom}_{\Omega^{\times k}} ({\bf k}[C_k] \otimes \mathcal{A}^{\otimes k} , \mathcal{A})
\subset \mathrm{Hom}_{\Omega^{\times k}} ({\bf k}[C_k] \otimes \mathcal{A}^{\otimes k} , \mathcal{A})$ given by
\begin{align*}
\mathrm{FamHom}_{\Omega^{\times k}} ({\bf k}[C_k] \otimes \mathcal{A}^{\otimes k} , \mathcal{A}) := \big\{  \eta^k =& (\eta^{k,[1]}, \ldots, \eta^{k, [k]}) \in  \mathrm{Hom}_{\Omega^{\times k}} ({\bf k}[C_k] \otimes \mathcal{A}^{\otimes k} , \mathcal{A}) ~\big|~\\ & \eta^{k, [r]}_{\alpha_1, \ldots, \alpha_k} \text{ does not depend on } \alpha_r, \text{ for } 1 \leq r \leq k \big\}.
\end{align*}

\begin{defn}
A {\bf ${Dend}_\infty$-family algebra} is a graded ${\bf k}$-module $\mathcal{A}$ together with a collection $\{ \eta^1, \eta^2, \ldots \}$, where
\begin{align*}
\eta^k = (  \eta^{k,[1]}, \ldots, \eta^{k, [k]}  ) \in \mathrm{FamHom}_{\Omega^{\times k}} ({\bf k}[C_k] \otimes \mathcal{A}^{\otimes k} , \mathcal{A}) \text{ with } \mathrm{deg}(\eta^{k, [r]}_{\alpha_1, \ldots, \alpha_k}) = k-2, \text{ for all } k \geq 1
\end{align*}
satisfying the following set of identities
\begin{align}\label{dend-fam-inf}
\sum_{m+n=N+1} \sum_{i=1}^m {\pm} ~\eta^{m, R_0 (m;1, \ldots, n, \ldots, 1)[r]}_{\alpha_1, \ldots, \alpha_i \cdots \alpha_{i+n-1}, \ldots, \alpha_{N}} \big(    a_1, \ldots, a_{i-1}, \eta^{n, R_i (m;1, \ldots, n, \ldots, 1)[r]}_{\alpha_i, \ldots, \alpha_{i+n-1}} (a_i, \ldots a_{i+n-1}), a_{i+n}, \ldots, a_{N} \big) = 0,
\end{align}
for homogeneous elements $a_1, \ldots, a_N \in \mathcal{A}$; $\alpha_1, \ldots, \alpha_N \in \Omega$; $[r] \in C_N$ and $N \geq 1$.
\end{defn}

When $\Omega$ is singleton, we get the notion of $Dend_\infty$-algebra considered in \cite{lod-val-book,A4}. Let $\mathcal{A} = \oplus_{i \in \mathbb{Z}} \mathcal{A}^i$ be a graded ${\bf k}$-module. Then $\mathcal{A} \otimes {\bf k}\Omega = \oplus_{i \in \mathbb{Z}} (\mathcal{A} \otimes {\bf k}\Omega)^i$ is a graded ${\bf k}$-module, where $(\mathcal{A} \otimes {\bf k}\Omega)^i = \mathcal{A}^i \otimes {\bf k} \Omega$ for any $i \in \mathbb{Z}$. In the next result, we show that a ${Dend}_\infty$-family algebra structure on a graded ${\bf k}$-module $\mathcal{A}$ induces an ordinary ${Dend}_\infty$-algebra structure on $\mathcal{A} \otimes {\bf k} \Omega$. This generalizes Proposition \ref{dend-fam-dend} in the homotopy context.

\begin{thm}
Let $(\mathcal{A}, \{ \eta^1, \eta^2, \ldots \})$ be a ${Dend}_\infty$-family algebra. Then $(\mathcal{A} \otimes {\bf k} \Omega, \{\overline{\eta}^1, \overline{\eta}^2, \ldots \} )$ is a ${Dend}_\infty$-algebra, where
\begin{align*}
\overline{\eta}^k : {\bf k}[C_k ] \otimes (\mathcal{A} \otimes {\bf k} \Omega)^{\otimes k} \rightarrow (\mathcal{A} \otimes {\bf k} \Omega), ~ \overline{\eta}^{k, [r]} (a_1 \otimes \alpha_1, \ldots, a_k \otimes \alpha_k) = \eta^{k, [r]}_{\alpha_1, \ldots, \alpha_k} (a_1, \ldots, a_k) \otimes \alpha_1 \cdots \alpha_k,
\end{align*}
for $a_1 \otimes \alpha_1, \ldots, a_k \otimes \alpha_k \in \mathcal{A} \otimes {\bf k} \Omega$ and $[r] \in C_k$.
\end{thm}

\begin{proof}
For any $a_i \otimes \alpha_i \in \mathcal{A} \otimes {\bf k} \Omega$ $(1 \leq i \leq N)$ and $[r] \in C_N$, we have
\begin{align*}
&\overline{\eta}^{m, R_0 (m;1, \ldots, n, \ldots, 1)[r]} \big( a_1 \otimes \alpha_1, \ldots, \overline{\eta}^{n, R_i (m;1, \ldots, n, \ldots, 1)[r]} (  a_i \otimes \alpha_i , \ldots, a_{i+n-1} \otimes \alpha_{i+n-1} ), \ldots, a_N \otimes \alpha_N \big) \\
&= \overline{\eta}^{m, R_0 (m;1, \ldots, n, \ldots, 1)[r]} \big( a_1 \otimes \alpha_1, \ldots, \eta^{n, R_i (m;1, \ldots, n, \ldots, 1)[r]}_{\alpha_i, \ldots, \alpha_{i+n-1}} (a_i, \ldots, a_{i+n-1}) \otimes \alpha_i \cdots \alpha_{i+n-1} , \ldots, a_N \otimes \alpha_N \big) \\
&= \eta^{m, R_0 (m;1, \ldots, n, \ldots, 1)[r]}_{\alpha_1, \ldots, \alpha_i \cdots \alpha_{i+n-1}, \ldots, \alpha_N} \big(  a_1, \ldots, \eta^{n, R_i (m;1, \ldots, n, \ldots, 1)[r]}_{\alpha_i, \ldots, \alpha_{i+n-1}} (a_i, \ldots, a_{i+n-1}), \ldots, a_N \big).
\end{align*}
This shows that $\{ \overline{\eta}^1, \overline{\eta}^2, \ldots \}$ satisfy the ${Dend}_\infty$-algebra identities as $\{ \eta^1, \eta^2, \ldots \}$ satisfy the $Dend_\infty$-family identities (\ref{dend-fam-inf}). This proves the result.
\end{proof}

\begin{prop}
Let $(\mathcal{A}, \{ \eta^1, \eta^2, \ldots \})$ be a ${Dend}_\infty$-family algebra. Then $(\mathcal{A}, \{ \mu^1, \mu^2, \ldots \})$ is an $A_\infty$-algebra relative to $\Omega$, where
\begin{align*}
\mu^k := \eta^{k, [1]} + \cdots + \eta^{k, [k]}, \text{ for all } k \geq 1.
\end{align*}
\end{prop}

\begin{proof}
Since $\{ \eta^1, \eta^2, \ldots \}$ is a $Dend_\infty$-family structure on $\mathcal{A}$, the identity (\ref{dend-fam-inf}) holds for all $[r] \in C_N$ with $N \geq 1$. For a fixed $N \geq 1$, by adding the $N$ many identities corresponding to each $[r] \in C_N$, we simply get
\begin{align*}
\sum_{m+n=N+1} \sum_{i=1}^m {\pm}~ \eta^{m, [1]+ \cdots + [m]}_{\alpha_1, \ldots, \alpha_i \cdots \alpha_{i+n-1}, \ldots, \alpha_N} \big(  a_1, \ldots, \eta^{n, [1] + \cdots + [n]}_{\alpha_i, \ldots, \alpha_{i+n-1}} (a_i, \ldots, a_{i+n-1}), \ldots, a_N  \big) = 0.
\end{align*}
Thus, we obtain the identity (\ref{inf-rel-om}) for $A_\infty$-algebra relative to $\Omega$ for the operations $\{ \mu^1, \mu^2, \ldots \}$. Hence the proof.
\end{proof}

 In the following, we introduce a notion of homotopy Rota-Baxter family on an $A_\infty$-algebra as a generalization of the Rota-Baxter family. We will show that a homotopy Rota-Baxter family induces a $Dend_\infty$-family structure.
 
 \begin{defn}
 Let $(\mathcal{A}, \{ \mu^1, \mu^2, \ldots \})$ be an $A_\infty$-algebra. A {\bf homotopy Rota-Baxter family} on $\mathcal{A}$ is a collection $\{ \mathcal{R}_\alpha : \mathcal{A} \rightarrow \mathcal{A} \}_{\alpha \in \Omega}$ of ${\bf k}$-linear maps of degree $0$ satisfying
 \begin{align}\label{hom-rb-fam}
 \mu^k \big(  \mathcal{R}_{\alpha_1} (a_1), \ldots, \mathcal{R}_{\alpha_k} (a_k) \big) = \sum_{r=1}^k \mathcal{R}_{\alpha_1 \cdots \alpha_k} \big( \mu^k \big( \mathcal{R}_{\alpha_1} (a_1), \ldots, a_r, \ldots, \mathcal{R}_{\alpha_k} (a_k)  \big)   \big), 
 \end{align}
for $a_i \in \mathcal{A}, \alpha_i \in \Omega$ and $k \geq 1.$
 \end{defn}

\begin{remark}
If the $A_\infty$-algebra $(\mathcal{A}, \{ \mu^1, \mu^2, \ldots \})$  is concentrated in arity $0$, then we have seen that $\mathcal{A} = \mathcal{A}^0$ is nothing but an ordinary associative algebra (i.e., $\mu^k = 0$ for $k \neq 2$). In this case, a homotopy Rota-Baxter family is nothing but a Rota-Baxter family.
\end{remark}
 
 \begin{thm}
 Let  $(\mathcal{A}, \{ \mu^1, \mu^2, \ldots \})$  be an $A_\infty$-algebra and $\{ \mathcal{R}_\alpha : \mathcal{A} \rightarrow \mathcal{A} \}_{\alpha \in \Omega}$ be a homotopy Rota-Baxter family. Then $(\mathcal{A}, \{ \eta^1, \eta^2, \ldots \})$ is a $Dend_\infty$-family algebra, where
\begin{align*}
\eta^k \in \mathrm{FamHom}_{\Omega^{\times k}} (\mathbf{k}[C_k] \otimes \mathcal{A}^{\otimes k}, \mathcal{A}), ~ \eta^{k, [r]}_{\alpha_1, \ldots, \alpha_k} (a_1, \ldots, a_k ) := \mu^k \big( \mathcal{R}_{\alpha_1} (a_1), \ldots, a_r, \ldots, \mathcal{R}_{\alpha_k} (a_k)   \big), 
\end{align*} 
for $[r] \in C_k$, $\alpha_i \in \Omega$ and $a_i \in \mathcal{A}$.
\end{thm}

\begin{proof}
Since $\{ \mathcal{R}_\alpha : \mathcal{A} \rightarrow \mathcal{A} \}_{\alpha \in \Omega}$ is a homotopy Rota-Baxter family, it follows from (\ref{hom-rb-fam}) that
\begin{align*}
 \mu^k \big(  \mathcal{R}_{\alpha_1} (a_1), \ldots, \mathcal{R}_{\alpha_k} (a_k) \big)  = \mathcal{R}_{\alpha_1 \cdots \alpha_k} \big( \eta^{k, [1] + \cdots + [k]}_{\alpha_1, \ldots, \alpha_k}  (a_1, \ldots, a_k) \big), \text{ for } k \geq 1.
\end{align*}
Moreover, since $(\mathcal{A}, \{ \mu^1, \mu^2, \ldots \})$ is an $A_\infty$-algebra, we have
\begin{align}\label{double-sum}
\sum_{m+n=N+1} \sum_{i=1}^m \pm~ \mu^m \big( a_1, \ldots, a_{i-1}, \mu^n (a_i, \ldots, a_{i+n-1}), \ldots, a_N   \big) = 0.
\end{align}
In this identity, let us replace the tuple $(a_1, \ldots, a_N)$ of elements of $\mathcal{A}$ by $(\mathcal{R}_{\alpha_1} (a_1), \ldots, a_r, \ldots, \mathcal{R}_{\alpha_N}(a_N))$, for some $1 \leq r \leq N$. We will see that we obtain $Dend_\infty$-family algebra identities. For any fixed $m, n, i$ with $m+n = N+1$ and $1 \leq i \leq m$, if $r \leq i-1$, then the term insider the double summation of (\ref{double-sum}) looks
\begin{align}\label{same}
&\mu^m \big(  \mathcal{R}_{\alpha_1} (a_1), \ldots, a_r, \ldots, \mathcal{R}_{\alpha_{i-1}} (a_{i-1}), \mu^n \big( \mathcal{R}_{\alpha_i} (a_i), \ldots, \mathcal{R}_{\alpha_{i+n-1}} (a_{i+n-1})  \big), \ldots, \mathcal{R}_{\alpha_N}(a_N)   \big) \nonumber \\
&= \mu^m \big( \mathcal{R}_{\alpha_1} (a_1), \ldots, a_r, \ldots, \mathcal{R}_{\alpha_{i-1}} (a_{i-1}), \mathcal{R}_{\alpha_i \cdots \alpha_{i+n-1}} (\eta^{n, [1] +\cdots+[n]}_{\alpha_i, \ldots, \alpha_{i+n-1}} (a_i, \ldots, a_{i+n-1}) ), \ldots, \mathcal{R}_{\alpha_N}(a_N)  \big) \nonumber \\
&= \mu^{m, [r]}_{\alpha_1, \ldots, \alpha_i \cdots \alpha_{i+n-1}, \ldots, \alpha_N} \big(   a_1, \ldots, a_{i-1}, \eta^{n, [1] +\cdots+[n]}_{\alpha_i, \ldots, \alpha_{i+n-1}} (a_i, \ldots, a_{i+n-1}), \ldots, a_N \big) \nonumber \\
&= \eta^{m, R_0 (m;1, \ldots, n, \ldots, 1) [r]}_{ \alpha_1, \ldots, \alpha_i \cdots \alpha_{i+n-1}, \ldots, \alpha_N } \big( a_1, \ldots, a_{i-1},  \eta^{n, R_i (m;1, \ldots, n, \ldots, 1)[r]}_{\alpha_i, \ldots, \alpha_{i+n-1}} (a_i, \ldots, a_{i+n-1})   , \ldots, a_N \big).
\end{align}
(This is precisely the term {inside the} double summation of the $Dend_\infty$-family identity (\ref{dend-fam-inf})). Similarly, if $r$ lies in either of the intervals $i \leq r \leq i+n-1$ or $i+n \leq r \leq N$, then the term inside the double summation of (\ref{double-sum}) looks the same as (\ref{same}). Hence we get the $Dend_\infty$-family algebra identity (\ref{dend-fam-inf}) for any $[r] \in C_N$. This completes the proof.
\end{proof}

\medskip

\medskip

%.................................................

\noindent {\bf Acknowledgements.} The author would like to thank Indian Institute of Technology (IIT) Kharagpur for providing the beautiful academic atmosphere where the research has been carried out.

\end{document}